\documentclass{amsart}

\usepackage{amscd,amsfonts,amsmath,amssymb,amsthm}
\usepackage{epsf,verbatim,xy}
\usepackage[mathscr]{eucal}
\usepackage[dvips]{graphicx}
\usepackage[latin5]{inputenc}
\usepackage{bbm,bm}
\usepackage{stmaryrd}
\usepackage{latexsym}
\usepackage{yfonts}
\usepackage{mathrsfs}
\usepackage{mathtools}
\usepackage{color,graphicx}
\usepackage{epstopdf}
\usepackage{graphics}
\usepackage{subfigure}

\makeatletter
\def\@seccntformat#1{\@ifundefined{#1@cntformat}%
    {\csname the#1\endcsname\quad}
    {\csname #1@cntformat\endcsname}}
\newcommand{\section@cntformat}{\S\thesection\quad}
\newcommand{\subsection@cntformat}{\S\S\thesubsection\quad}
\makeatletter

\usepackage{etoolbox}
\makeatletter
\patchcmd{\section}{\scshape}{\bf}{}{}
\makeatother

\allowdisplaybreaks

\newtheorem{Thm}{Theorem}
\newtheorem{Lem}{Lemma}

\newtheorem{Prop}{Proposition}
\theoremstyle{remark}
\theoremstyle{remark}
\theoremstyle{remark}
\numberwithin{equation}{section}

\def\sumprime_#1{\setbox0=\hbox{$\scriptstyle{#1}$}
\setbox2=\hbox{$\displaystyle{\sum}$}
\setbox4=\hbox{${}'\mathsurround=0pt$}
\dimen0=.5\wd0 \advance\dimen0 by-.5\wd2
\ifdim\dimen0>0pt
\ifdim\dimen0>\wd4 \kern\wd4 \else\kern\dimen0\fi\fi
\mathop{{\sum}'}_{\kern-\wd4 #1}}

\begin{document}

\title{The pseudoinverse of the Laplacian matrix: Asymptotic behavior of its trace}

\author{Fatih Ecevit \hspace{.3cm} Cem Yal\c{c}{\i}n Y{\i}ld{\i}r{\i}m} 
\address{Department of Mathematics, Bo\~{g}az{\tiny \.{I}}\c{c}{\tiny \.{I}}  University, \.{I}stanbul 34342, Turkey}
\email{fatih.ecevit@boun.edu.tr, yalciny@boun.edu.tr} 
\date{\today}


\begin{abstract}
In this paper we are concerned with the asymptotic behavior of
\[
	\operatorname{tr}(\mathcal{L}^+_{\rm sq})
	= \frac{1}{4} \sum_{j,k=0 \atop (j,k) \neq (0,0)}^{n-1} \frac{1}{1-\frac{1}{2} \big( \cos \frac{2\pi j}{n} + \cos \frac{2\pi k}{n} \big)},
\]
the trace of the pseudoinverse of the Laplacian matrix related with the square lattice, as $n \to \infty$.
The method we developed for such sums in former papers depends on the use of Taylor approximations for the summands. It was shown
that the error term depends on whether the Taylor polynomial used is of degree two or higher. Here we carry this out for the square lattice
with a fourth degree Taylor polynomial and thereby obtain a result with an improved error term which is perhaps the most precise
one can hope for.
\end{abstract}

\begingroup
\def\uppercasenonmath#1{} 
\let\MakeUppercase\relax 
\maketitle \markright{The pseudoinverse of the Laplacian matrix: Asymptotic behavior of its trace}
\endgroup

\let\MakeUppercase\relax

\section{Introduction}

This paper is a continuation of our articles \cite{BoysalEcevitYildirim18} and
\cite{BoysalEcevitYildirim20} where certain sums involving the cosine function were asymptotically
evaluated over the triangular, the square and the modified union jack lattices. 
Here we present the most precise calculation possible
within the framework set up in \cite{BoysalEcevitYildirim20} in the case of the square lattice.
The sums studied in \cite{BoysalEcevitYildirim18} and
\cite{BoysalEcevitYildirim20} are realized as the trace $\operatorname{tr}({\mathcal{L^+}})$
of the pseudoinverse $\mathcal{L}^+$ of the Laplacian matrix $\mathcal{L}$ which is a
fundamental object in spectral theory of graphs, networks,
grids, and arithmetic of curves (see e.g. \cite{Hu,Cj,Poz,Gu,C1,C2,BGPWZ}). 
In certain cases (see \cite{C1,C2}), $\operatorname{tr}({\mathcal{L^+}})$ arises as the only nontrivial
term in the calculation of intrinsic graph invariants such as the Kirchhoff index and the tau constant.

%

For a variety of lattices endowed with periodic boundary conditions, $\operatorname{tr}({\mathcal{L^+}})$ can be expressed in terms of a sum of the form \cite{BoysalEcevitYildirim20}
\begin{equation} \label{eq:Fn-concrete}
	F_n= \!\!\! \sum_{j,k=0 \atop (j,k)\neq (0,0)}^{n-1} \frac{1}{1- \dfrac{1}{L}\sum\limits_{\ell = 1}^{L} \cos (\mathbf{s}_{\ell} \cdot \mathbf{t}_{j,k})}
\end{equation}
where $L \ge 2$ is an integer, $\mathbf{s}_{\ell} = (s_{1,\ell},s_{2,\ell}) \in \mathbb{Z}^2 \backslash \{ \mathbf{0} \}$ for
$\ell = 1,\ldots,L$ with $\mathbf{s}_{1} = (1,0)$ and $\mathbf{s}_{2} = (0,1)$, and $\boldsymbol{t}_{j,k} = (t_j,t_k) = (\frac{2\pi j}{n}, \frac{2\pi k}{n})$
for $j,k \in \mathbb{Z}$. For instance, for the square, triangular, and modified union
jack lattices, it is known that (see \cite{Poz})
\[
	\operatorname{tr}(\mathcal{L}^+_{\rm sq})
	= \frac{1}{4} F_n^{\rm sq} 
	= \frac{1}{4} \sum_{j,k=0 \atop (j,k)\neq (0,0)}^{n-1} \frac{1}{1-\frac{1}{2} (\cos \frac{2\pi j}{n} + \cos \frac{2\pi k}{n})},
\]
\[
	\operatorname{tr}(\mathcal{L}^+_{\rm tr})
	= \frac{1}{6} F_n^{\rm tr} 
	= \frac{1}{6} \sum_{\substack{j,k=0\\(j,k)\neq (0,0)}}^{n-1} 
	\frac{1}{1- \frac{1}{3}(\cos{\frac{2\pi j}{n}}+ \cos{\frac{2\pi k}{n}} + \cos{\frac{2\pi (j+k)}{n}})},
\]
\[
	\operatorname{tr}(\mathcal{L}^+_{\rm muj})
	= \frac{1}{8} F_n^{\rm muj} 
	= \frac{1}{8} \!\!\!\!\! \sum_{j,k=0 \atop (j,k)\neq (0,0)}^{n-1} \!\!\!\!
	\frac{1}{1-\frac{1}{4}(\cos \frac{2\pi j}{n}+\cos \frac{2\pi k}{n}+\cos\frac{2\pi (j-k)}{n}+\cos\frac{2\pi (j+k)}{n})}.
\]
Former studies on $F_n$ were either outright wrong (\cite{Ye} gives purported approximate values for divergent integrals) or 
quite rough (the estimates of \cite{C1} do not even capture the asymptotic value of $F_n$) - whence arose the need to do correct and precise calculations.  
The asymptotic behavior of the sum $F_n^{\rm tr}$ associated with the triangular lattice was studied in \cite{BoysalEcevitYildirim18}.
There it was shown that
\begin{align}
	F_n^{\rm tr} \label{eq:A2}
	& = \frac{\sqrt{3}}{\pi}n^2\log n
	+ \frac{\sqrt{3}}{\pi} \Big(\gamma + \log \big( \frac{4 \pi \sqrt[4]{3}}{\Gamma(\frac{1}{3})^3}\big) \Big) n^2
	+ \mathcal{O}(\log n),
	\quad
	\text{as }
	n \to \infty,
\end{align}
where $\gamma$ is Euler's constant. The approach in \cite{BoysalEcevitYildirim18} was generalized in \cite{BoysalEcevitYildirim20}
to develop methods for obtaining the asymptotic behavior of the general sum $F_n$ of \eqref{eq:Fn-concrete} within errors of $\mathcal{O}(\log n)$
and $\mathcal{O}(1)$ as $n \to \infty$. They are based on the asymptotic analyses of integrals of $f$ and sums and integrals of $f_m$ where
\begin{equation} \label{eq:fandpsi}
	f(\mathbf{x}) = \frac{1}{\psi(\mathbf{x})}
	\quad
	\text{with}
	\quad
	\psi(\mathbf{x})
	= 1 - \frac{1}{L} \sum_{\ell = 1}^{L} \cos (\mathbf{s}_{\ell} \cdot \mathbf{x}),
	\qquad
	\mathbf{x} = (x,y) \in \mathbb{R}^2,
\end{equation}
and
\begin{equation} \label{eq:fmandpm}
	f_m(\mathbf{x}) = \frac{1}{p_m(\mathbf{x})}
	\quad
	\text{with}
	\quad
	p_{m}(\mathbf{x})
	= \frac{1}{L} \sum_{\ell = 1}^{L}
	\sum_{j=1}^{m} \frac{(-1)^{j+1}(\mathbf{s}_{\ell} \cdot \mathbf{x})^{2j}}{(2j)!}
	\quad
	(m \ge 1);
\end{equation}
$p_{m}$ is the $2m$-th order Taylor polynomial approximation of $\psi$ around the origin.
As shown in \cite{BoysalEcevitYildirim20}, the use of $f_m$ allows for the determination of the asymptotic expansion with an error term of $\mathcal{O}(\log n)$ for $m = 1$,
and $\mathcal{O}(1)$ for any larger value of $m$. The ideal choice is therefore $f_2$. However, working with $f_2$ demands significantly more delicate analyses compared
to $f_1$. In fact, the examples provided in \cite{BoysalEcevitYildirim20} are based on the use of $f_1$ for proving
\begin{equation} \label{eq:Fnfsqold}
	F_n^{\rm sq} = \frac{2}{\pi}n^2\log n+\frac{2}{\pi}\Big(\gamma+\log \big( \frac{4 \sqrt{2\pi}}{\Gamma(\frac{1}{4})^2} \big) \Big) n^2+\mathcal{O}(\log n),
\end{equation}
and
\begin{align} \label{eq:B2}
	F_n^{\rm muj}
	& = \frac{4}{3\pi}n^2\log n
	+ \frac{4}{3\pi} \Big(\gamma + \log \big( \frac{4 \sqrt{6\pi}}{\Gamma(\frac{1}{4})^2}\big) \Big) n^2
	+ \mathcal{O}(\log n),
\end{align}
as $n \to \infty$. In this paper we use $f_2$ for the first time and prove the following:

\begin{Thm} \label{thm:main}
As $n \to \infty$, we have
\begin{equation} \label{eq:Fnfsq}
	F_n^{\rm sq} = \frac{2}{\pi}n^2\log n+\frac{2}{\pi}\Big(\gamma+\log \big( \frac{4 \sqrt{2\pi}}{\Gamma(\frac{1}{4})^2} \big) \Big) n^2+\mathcal{O}(1).
\end{equation}
\end{Thm}

The motivation behind Theorem~\ref{thm:main} comes from numerical evidence.
Note that the asymptotic expansions \eqref{eq:A2}, \eqref{eq:Fnfsqold}, and \eqref{eq:B2} are all in the form
\[
	F_n = a n^2 \log n + b n^2 + \mathcal{O}(\log n)
\]
for some constants $a$ and $b$. The errors
\begin{equation} \label{eq:En}
	E_n = F_n - (a n^2 \log n + b n^2)
\end{equation}
displayed in Figure \ref{fig:test} suggest that, in fact,
\[
	F_n = a n^2 \log n + b n^2 + \mathcal{O}(1)
\]
with $E_n \approx -0.12$, $E_n \approx -0.25$, $E_n \approx -0.37$ as $n \to \infty$ for the sums associated with the square, triangular, and modified union
jack lattices respectively.
\begin{figure}[t]
	\centering
	\begin{subfigure}
		{\includegraphics[width=1\textwidth,trim={1.8cm 0.1cm 2.6cm 0.8cm},clip]{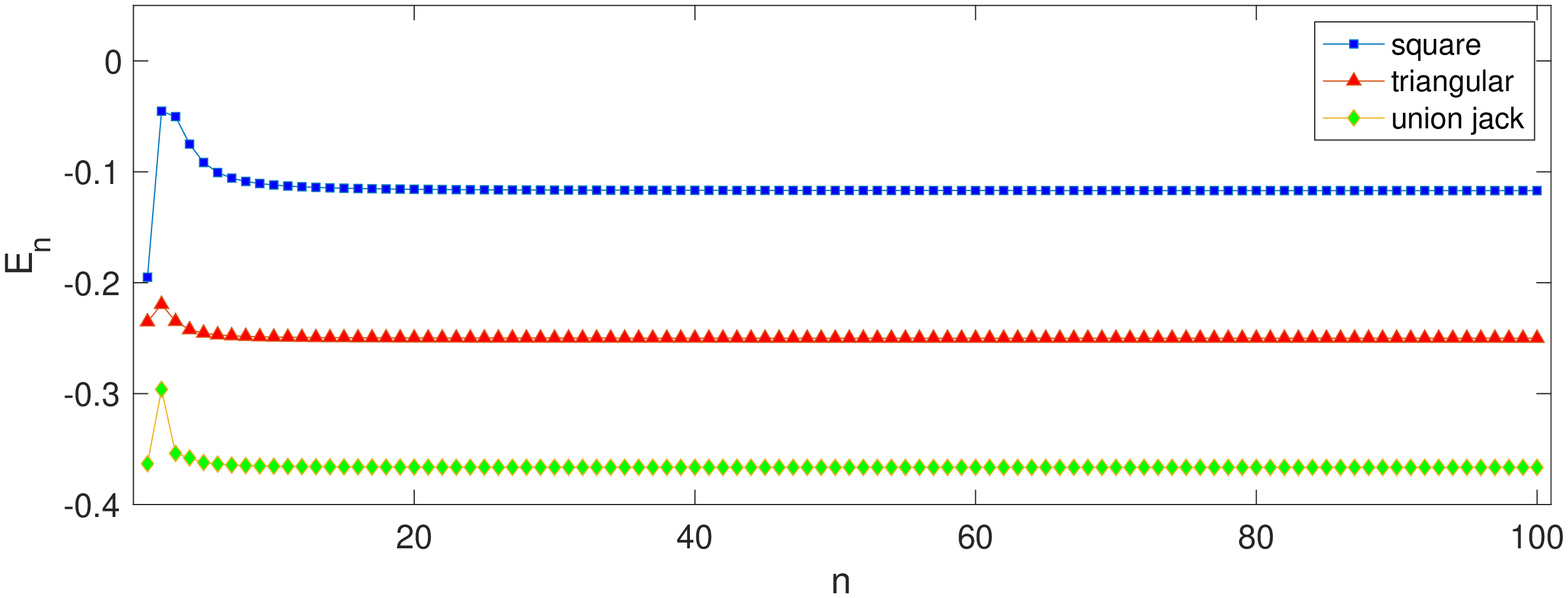}}
	\end{subfigure}
	\begin{subfigure}
		{\includegraphics[width=1\textwidth,trim={1.8cm 0.1cm 2.6cm 0.8cm},clip]{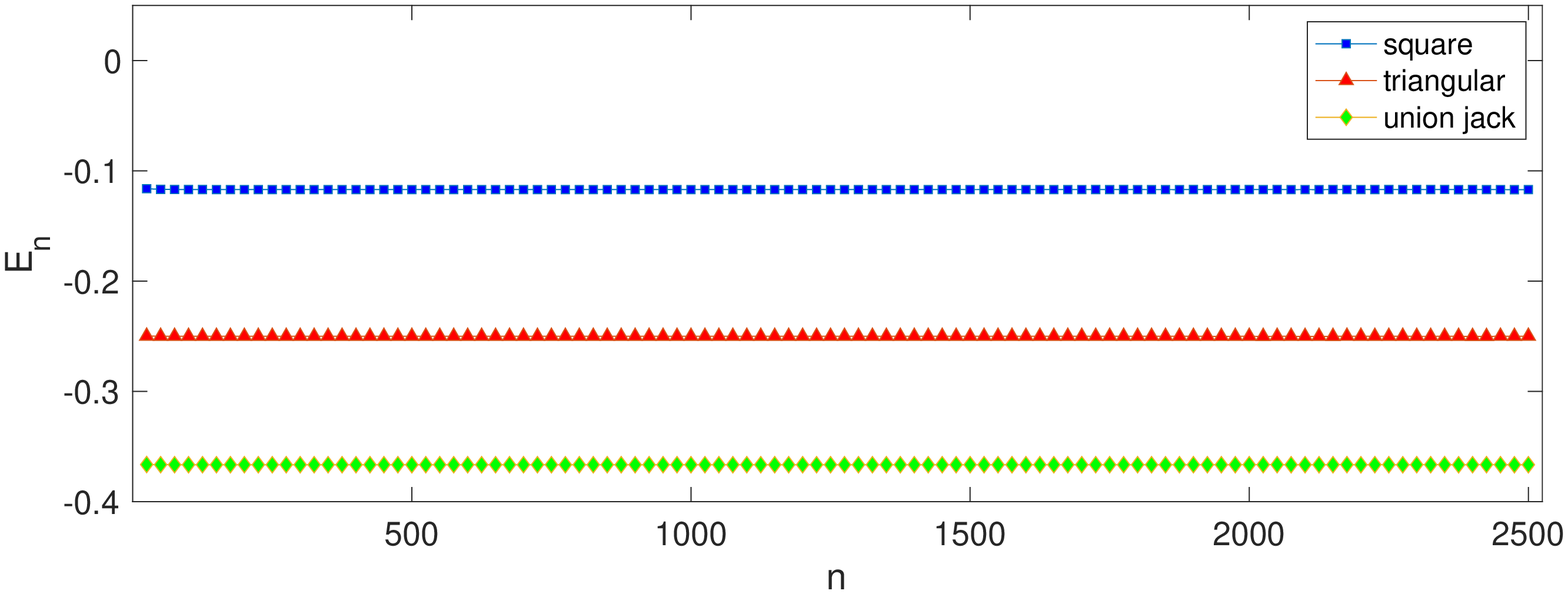}}
	\end{subfigure}
	\caption{$E_n$ in \eqref{eq:En} associated with the square, triangular, and modified union jack lattices for
	$n=1,\ldots,100$ on top and for $25 \le n \le 2500$ with increments of $25$ at the bottom.}
	\label{fig:test}
\end{figure}

As stated above, the estimate given in Theorem~\ref{thm:main} is as precise as it can get within 
the analytical framework developed in \cite{BoysalEcevitYildirim20}. Even if one may aspire to carry out
an exact algebraic calculation involving cyclotomic fields (at least in some special cases such as $n$ running through the sequence of primes, or
the sequence of powers of $2$), the question of
how one would obtain the value of $F_n$, or merely its asymptotic value, from the resulting algebraic numbers remains.


The paper is organized as follows. In \S2 we revisit the method developed in \cite{BoysalEcevitYildirim20} for studying the asymptotic behavior of the general
sum \eqref{eq:Fn-concrete} as $n \to \infty$. In the same section, we also take on the sum $F_n^{\rm sq}$ corresponding to the square lattice, and deferring the technical details to \S3 and \S4, present the proof of Theorem \ref{thm:main}. 
In these latter sections we study the asymptotic behavior of $I_n^{\beta}(f_2)$
and $F_n^{\beta}(f_2)$ (an integral and a sum related with $f_2$) when $F_n = F_n^{\rm sq}$.

\section{The setup}

As shown in \cite{BoysalEcevitYildirim20}, using the $2\pi$-periodicity of $\psi$ in both of its arguments, the sum $F_n$ in \eqref{eq:Fn-concrete} can be recast as
\[
	F_n = F_n (f) = \sum_{\mathbf{t}_{j,k} \in D_n} f(\mathbf{t}_{j,k})
	\quad
	\text{with}
	\quad
	D_n
	= \bigcup_{\mathbf{t}_{j,k} \in [-\pi,\pi)^2 \backslash \{ \mathbf{0} \}} X_{j,k}
\]
where $X_{j,k} = X_j \times X_k$ with $X_j = [x_j,x_{j+1}]$ and $x_j = t_j - \frac{\pi}{n}$ for $j \in \mathbb{Z}$. Explicitly
\[
	D_n
	= 	\left\{
		\begin{array}{ll}
			[-\pi,\pi]^2 \backslash [-\frac{\pi}{n},\frac{\pi}{n}]^2, & n \text{ odd},
			 \vspace{0.1cm} \\
			\left[ -\pi-\frac{\pi}{n},\pi-\frac{\pi}{n} \right]^2 \backslash \left[ -\frac{\pi}{n},\frac{\pi}{n} \right]^2, & n \text{ even}.
		\end{array}
	\right.	
\]
A first approximation
to $F_n(f)$ is the integral
\[
	I_{n}(f) = \dfrac{1}{\Delta_n^2} \iint_{D_n} f(\mathbf{x}) \, d\mathbf{x},
	\qquad \big( \Delta_n = \frac{2\pi}{n} \big),
\]
since $F_n(f)$ can be obtained from $I_n(f)$
by applying the product cubature rule, with both factors coming from the midpoint
rule, on each of the rectangles $X_{j,k}$ in $D_n$.
Other possible approximations are 
\[
	F_n (f_m) = \!\!\!\! \sum_{\mathbf{t}_{j,k} \in D_n} \!\!\! f_m(\mathbf{t}_{j,k}),
	\qquad
	I_{n}(f_m) = \dfrac{1}{\Delta_n^2} \! \iint_{D_n} \!\! f_m(\mathbf{x}) \, d\mathbf{x},
\]
and these are not problematic when $m=1$ since $p_1(\mathbf{x}) = 0$ only for $\mathbf{x} = \mathbf{0}$.
However, when $m>1$, $p_m$ may also vanish at some other points in $[-\pi,\pi] \times [-\pi,\pi]$, and therefore $D_n$
will have to be restricted to a smaller region (see \cite[Remark 1]{BoysalEcevitYildirim20}). Given a fixed $\beta \in (0,1)$,
$p_m$ ($m \ge 1$) is never zero on
\begin{equation} \label{eq:Dnbeta}
	D_{n}^{\beta} = \bigcup_{|t_{j}|, |t_k| \le \frac{\sqrt{5(1-\beta)}}{\overline{s}} \atop (t_j,t_k) \ne (0,0)} X_{j,k}
	\qquad
	\big( \overline{s} = \max_{1 \le \ell \le L} \Vert \mathbf{s}_{\ell} \Vert \big),
\end{equation}
and 
\[
	F_n^{\beta} (f_m) = \sum_{\mathbf{t}_{j,k} \in D_n^{\beta}} f_m(\mathbf{t}_{j,k}),
	\qquad
	I_n^{\beta}(f_m) = \frac{1}{\Delta_n^2} \iint_{D_n^{\beta}} f_m(\mathbf{x}) d\mathbf{x}
\]
provide alternative approximations to $F_n(f)$. The main results of \cite{BoysalEcevitYildirim20} for the determination of the
asymptotic behavior of $F_n(f)$ are Theorems A and B:
\vskip.2cm
\noindent
{\bf Theorem A.}\;\;
\emph{If $\Sigma_n \in \{ F_n(f), F_n^{\beta} (f), F_n(f_1), F_n^{\beta} (f_m) , 
I_n(f), I_n^{\beta}(f), I_n(f_1), I_n^{\beta}(f_m) \}$ for some fixed
$\beta \in (0,1)$ and $m \ge 1$, then
\[
	\Sigma_n = \frac{|\Phi|}{\pi \sqrt{\det(S^TS)}} \, n^2 \log n + \mathcal{O}(n^2)
\]
as $n \to \infty$ where
\[
	S =
	\begin{bmatrix}
		s_{1,1} & \cdots & s_{1,\ell} & \cdots & s_{1,|\Phi|}
		\\
		s_{2,1} & \cdots & s_{2,\ell} & \cdots & s_{2,|\Phi|}
	\end{bmatrix}^T.
\]}

\noindent
{\bf Theorem B.}\;\;
\emph{As $n \to \infty$, we have
\begin{equation} \label{eq:simpleOlog}
	F_n(f) - I_n(f) + I_n(f_1) - F_n(f_1)
	= \mathcal{O}(\log n),
\end{equation}
and, for any fixed $\beta \in (0,1)$,
\begin{equation} \label{eq:Olog-O1}
	F_n(f) - I_n(f) + I_n^{\beta}(f_m) - F_n^{\beta}(f_m) = 
	\left\{ 	
		\begin{array}{cl}
			\mathcal{O}(\log n), & m = 1,
			\\
			\mathcal{O}(1), & m \ge 2.	
		\end{array}
	\right.
\end{equation}}
\vskip.2cm
\noindent The asymptotic behavior of $I_n(f_1)$ and $F_n(f_1)$ were derived in the general setting of the sum \eqref{eq:Fn-concrete} with respective error terms of
$\displaystyle{\mathcal{O}(\frac{1}{n^2})}$ and $\displaystyle{\mathcal{O}(\frac{\log n}{n^2})}$ in \cite{BoysalEcevitYildirim20}. That of the integral $I_n(f)$ 
was derived up to an error term
of $\mathcal{O}(1)$ when $F_n(f)$ pertains to the triangular lattice in \cite{BoysalEcevitYildirim18}, and when it corresponds to the square and modified
union jack lattices in \cite{BoysalEcevitYildirim20}. Using \eqref{eq:simpleOlog}, the asymptotic behavior of $F_n(f)$ in these three cases 
were obtained within errors of $\mathcal{O}(\log n)$.

As is apparent from \eqref{eq:Olog-O1}, however, the ideal choice for the methods developed in \cite{BoysalEcevitYildirim20} is $m=2$
as it yields the optimal error bound of $\mathcal{O}(1)$ with the least possible effort. On the other hand, working with $I_n^{\beta}(f_2)$ and $F_n^{\beta}(f_2)$ is significantly more challenging in the general setting of the sum \eqref{eq:Fn-concrete} when compared to $I_n(f_1)$ and $F_n(f_1)$. This is because, in addition to the significantly more difficult analyses, the former demands working with the factorization of a polynomial of degree four (namely $p_2$) so as to obtain a partial fraction decomposition of $f_2$ to begin with, whereas the latter involves $p_1$ which is a polynomial of degree two only.

With this in mind, in this paper we study the asymptotic behavior of $I_n^{\beta}(f_2)$ and $F_n^{\beta}(f_2)$ associated with the
sum $F_n = F_n(f) = F_n^{\rm sq}$ corresponding to the square lattice. Let us note that in this case (cf. \eqref{eq:fandpsi}, \eqref{eq:fmandpm},
and \eqref{eq:Dnbeta})
\[
	f(\mathbf{x}) = \dfrac{1}{\psi(\mathbf{x})} = \dfrac{1}{1-\frac{1}{2} (\cos {x} + \cos {y})},
\]
\[
	f_1(\mathbf{x}) = \dfrac{1}{p_1(\mathbf{x})} = \dfrac{4}{x^2+y^2},
	\qquad
	f_2(\mathbf{x}) = \dfrac{1}{p_2(\mathbf{x})} = \dfrac{4}{x^2+y^2 - \frac{1}{12} (x^4+y^4)},
\]
and $\overline{s} = 1$. For the calculations and analyses that follow, we choose
\[
	\beta = 1- \frac{\pi^2}{20} \in (0,1)
\]
so that
\begin{equation} \label{eq:Dneps}
	D_{n}^{\beta}
	= \bigcup_{|t_{j}|, |t_k| \le \frac{\pi}{2} \atop (t_j,t_k) \ne (0,0)} X_{j,k}
	= [-\beta_n,\beta_n]^2 \backslash (-\frac{\pi}{n},\frac{\pi}{n})^2
\end{equation}
where
\begin{equation} \label{eq:n0}
	\beta_n = \frac{\pi}{2} \left( 1 + \frac{2-n_0}{n} \right)
	\qquad
	\text{with}
	\qquad
	n \equiv n_0 \ (\bmod 4),
	\quad
	n_0 \in \{0,1,2,3\}.
\end{equation}
Observe that the second identity in \eqref{eq:Dneps} is a consequence of
\begin{equation} \label{eq:explainthis}
	|t_j| \le \dfrac{\pi}{2}
	\
	\Leftrightarrow
	\
	|j| \le \dfrac{n-n_0}{4} 
	\
	\Leftrightarrow
	\quad
	-\dfrac{2\pi}{n} \, \dfrac{n-n_0}{4} - \dfrac{\pi}{n} \le x_j \le \dfrac{2\pi}{n} \, \dfrac{n-n_0}{4} - \dfrac{\pi}{n}.
\end{equation}

\noindent
\emph{Proof of Theorem \ref{thm:main}.} As was shown in \cite{BoysalEcevitYildirim20}, for the square lattice we have
\begin{equation} \label{eq:Inf}
	I_n(f) = \frac{2}{\pi} n^2 \log n + \frac{1}{\pi} \big( \log \big( \frac{8}{\pi^2} \big) + \frac{4G}{\pi} \big)n^2 + \mathcal{O}(1)
\end{equation}
where $G$ is Catalan's constant. In \S3, we show that 
\begin{equation} \label{eq:Inf2}
	I_n^{\beta}(f_2) = a_0 \, n^2 \log n + a_1 \, n^2 + a_2 \, n + \mathcal{O}(1)
\end{equation}
by explicitly determining the constants $a_j$. In \S4, we prove that
\begin{equation} \label{eq:Fnf2}
	F_n^{\beta}(f_2) = b_0 \, n^2 \log n + b_1 \, n^2 + b_2 \, n + \mathcal{O}(1)
\end{equation}
without making the constants $b_j$ explicit. (The proof of \eqref{eq:Fnf2} given in \S4 is based on a decomposition of $F_n(f_2)$
into six pieces followed by a study of their asymptotic behavior. For three of them the coefficients will be given explicitly.) In fact,
we do not need the explicit values of these constants since the use of \eqref{eq:Inf}, \eqref{eq:Inf2}, and \eqref{eq:Fnf2}  
in \eqref{eq:Olog-O1} implies that
\[
	F_n(f) = c_0 \, n^2 \log n + c_1 \, n^2 + c_2 \, n + \mathcal{O}(1)
\]
for some constants $c_j$ and this, in turn, implies through \eqref{eq:Fnfsqold} that these constants must be as given in
\eqref{eq:Fnfsq}. Thus Theorem \ref{thm:main} follows.
\hfill $\square$

\section{Asymptotic behavior of $I_n^{\beta}(f_2)$}

In this section we prove the following for the asymptotic behavior of $I_n^{\beta}(f_2)$.
\begin{Prop} \label{prop:Inbetaf2asymp}
As $n \to \infty$, we have
\begin{multline*}
	I_n^{\beta}(f_2)
	= \frac{2}{\pi} n^2 \log n
	+ \frac{2}{\pi} 
	\Big(
		\frac{2G + \lambda}{\pi} + \log \frac{2\sqrt{6}}{\pi} 	
	\Big) n^2
	\\
	+ \frac{96}{\pi^2 \mu} 
	\Big(
		2 \sqrt{\frac{\nu+1}{\nu-1}} \arctan \sqrt{\frac{\nu+1}{\nu-1}} + \frac{1}{\sqrt{\nu}} \log \Big( \frac{\sqrt{\nu}+1}{\sqrt{\nu}-1} \Big)
	\Big) 
	(2-n_0) n
	+ \mathcal{O}(1),
\end{multline*}
where, denoting the Clausen function by $\operatorname{Cl}_2$,
\begin{multline*}
	\lambda
	= \operatorname{Cl}_2(2 \arctan \rho)
	- \operatorname{Cl}_2(\pi + 2 \arctan \rho)
	+ (\frac{\pi}{2} + 2 \arctan \rho) \, \log (\rho)
	\\
	- \operatorname{Cl}_2\big( \frac{\pi}{2} + \arccos(\frac{\nu-1}{\nu+1}) \big)
	- \operatorname{Cl}_2\big( \frac{\pi}{2} - \arccos(\frac{\nu-1}{\nu+1}) \big),
\end{multline*}
and
\[
	\mu = \sqrt{24^2+48\pi^2-\pi^4},
	\qquad
	\nu = \frac{\mu + 24}{\pi^2},
	\qquad
	\rho = \nu-\sqrt{\nu^2-1}.
\]
\end{Prop}
\noindent Note that $\mu \approx 30.86, \ \nu \approx 5.56, \ \rho \approx 0.09$.

The proof of Proposition \ref{prop:Inbetaf2asymp} is based on several lemmas.
\begin{Lem} \label{lemma:Inbetaf2asymp1}
As $n \to \infty$, we have
\begin{equation} \label{eq:Inf2asy2}
	I_n^{\beta}(f_2)
	= \frac{2}{\pi} n^2 \log n
	+ \frac{2}{\pi} \big( \frac{2G}{\pi} + \log \frac{\sqrt{6}}{\pi} - \frac{1}{\pi} \big( J_{n,1} + J_{n,2} \big) \big) n^2
	+ \mathcal{O}(1)
\end{equation}
with 
\begin{equation} \label{eq:Jn1Jn2}
	J_{n,1} = \int_{0}^{\pi/2} \log(2u_{n}-1 - \cos \theta) \, d\theta,
	\quad
	J_{n,2} = \int_{0}^{\pi/2} \log(\cos\theta +1 - \frac{1}{u_n}) \, d\theta
\end{equation}
where 
\[  
	u_{n} = \alpha_n+\sqrt{\alpha_n^2-\frac{1}{2}}
	\quad
	\text{with}
	\quad
	\alpha_n = \frac{\beta_n^2 +6}{2\beta_n^2}.
\]
\end{Lem}
\begin{proof}
Considering $I_n^{\beta}(f_1)$ first, switching to polar coordinates, we compute
\begin{align*}
	I_n^{\beta}(f_1)
	& = \frac{1}{\Delta_n^2} \iint_{D_n^{\beta}} f_1 \, dx \, dy
	= \frac{4}{\Delta_n^2} \iint_{D_n^{\beta}} \dfrac{dx \, dy}{x^2+y^2}
	= \frac{16}{\Delta_n^2} \iint\limits_{[0,\beta_n]^2 \backslash [0,\frac{\pi}{n}]^2}
	\dfrac{dx \, dy}{x^2+y^2}
	\\
	& = \frac{32}{\Delta_n^2} \int_{0}^{\frac{\pi}{4}} \int_{\frac{\pi}{n \, \cos \theta}}^{\frac{\beta_n}{\cos \theta}}
	\dfrac{1}{r} \, dr \, d\theta
	= \frac{8\pi}{\Delta_n^2} \, \log \frac{n \, \beta_n}{\pi},
\end{align*}
so that
\begin{align} \label{eq:Inf1asy}
	I_n^{\beta}(f_1)
	= \frac{2}{\pi} \, n^2 \log n - \frac{\log 4}{\pi} \, n^2 + \frac{2(2-n_0)}{\pi} \, n + \mathcal{O} \big( 1\big).
\end{align}
Similarly, writing $\eta^2(\theta) = \cos^{4} \theta + \sin^4 \theta$, we have 
\begin{align*}
	I_n^{\beta}(f_2)
	& = \frac{1}{\Delta_n^2} \iint_{D_n^{\beta}} f_2 \, dx \, dy
	= \frac{4}{\Delta_n^2} \iint_{D_n^{\beta}} \dfrac{dx \, dy}{x^2+y^2-\frac{1}{12}(x^4+y^4)}
	\\
	& = \frac{16}{\Delta_n^2} \iint\limits_{[0,\beta_n]^2 \backslash [0,\frac{\pi}{n}]^2}
	\dfrac{dx \, dy}{x^2+y^2-\frac{1}{12}(x^4+y^4)}
	= \frac{32}{\Delta_n^2} \int_{0}^{\frac{\pi}{4}} \int_{\frac{\pi}{n \, \cos \theta}}^{\frac{\beta_n}{\cos \theta}}
	\dfrac{dr \, d\theta}{r - \frac{\eta^2(\theta)}{12} r^3}
	\\
	& = \frac{32}{\Delta_n^2} \int_{0}^{\frac{\pi}{4}} \int_{\frac{\pi}{n \, \cos \theta}}^{\frac{\beta_n}{\cos \theta}}
	\Big( \frac{1}{r} + \dfrac{r \, \eta^2(\theta)}{12-r^2 \, \eta^2(\theta)} \Big) dr \, d\theta,
\end{align*}
and therefore
\begin{equation} \label{eq:Inf2asy1}
	I_n^{\beta}(f_2)
	= I_n^{\beta}(f_1)
	+ \frac{16}{\Delta_n^2} \int_{0}^{\frac{\pi}{4}}
	\Big( \log(12 - (\frac{\pi \, \eta(\theta)}{n \, \cos \theta})^2 ) - \log(12 - (\frac{\beta_n \eta(\theta)}{\cos \theta})^2) \Big) \, d\theta.
\end{equation}
We have
\begin{align}
	\frac{16}{\Delta_n^2} \int_{0}^{\frac{\pi}{4}} \log(12 - (\frac{\pi \, \eta(\theta)}{n \, \cos \theta})^2) d\theta
	& = \frac{16}{\Delta_n^2} \int_{0}^{\frac{\pi}{4}} \! \Big( \log 12
		+ \mathcal{O} \Big( \frac{1}{n^2}\Big) \Big) d\theta =   \frac{\log 12}{\pi} n^2 + \mathcal{O} (1).
	\label{eq:Int1}		
\end{align}
On the other hand, using 
\[
	\int_{0}^{\theta} \log(\cos\phi) \, d\phi
	= -\theta \log 2 + \frac{1}{2} \operatorname{Cl}_2(\pi - 2 \theta)
\]
(\cite[p.306, Formula 5]{L}), we have
\begin{align}
	\int_{0}^{\frac{\pi}{4}} \!
	\log\big( 12 - (\frac{\beta_n \eta(\theta)}{\cos \theta})^2 \big) d\theta
	& = \!\! \int_{0}^{\frac{\pi}{4}}
	\!\! \big[ \log(12 \cos^2 \theta - \beta_n^2 \eta^2(\theta)) - 2 \log (\cos\theta) \big] d\theta
	\label{eq:mid}
	\\
	& = \!\int_{0}^{\frac{\pi}{4}} \!\!\log(12 \cos^2 \theta - \beta_n^2 \eta^2(\theta)) \, d\theta
	+ \frac{\pi \, \log 2 - 2G}{2}.
	\nonumber
\end{align}
Since $\eta^2(\theta) = 2\cos^4 \theta -2\cos^2\theta + 1$, we have
\begin{align*}
	12 \cos^2 \theta - \beta_n^2 & \eta^2(\theta)
	= -2\beta_n^2 (\cos^4\theta -2\alpha_n\cos^2\theta+\frac{1}{2})
	\\
	& = \frac{\beta_n^2}{2}
	\big(2(\alpha_n+\sqrt{\alpha_n^2-\frac{1}{2}})-2\cos^2\theta\big)
	\big(2\cos^2\theta-2(\alpha_n-\sqrt{\alpha_n^2-\frac{1}{2}})\big)
	\\
	& = \frac{\beta_n^2}{2}
	\big(2u_n-1-\cos2\theta \big)
	\big(\cos2\theta+1-\frac{1}{u_n}\big),
\end{align*}
so that
\begin{align*}
	& \int_{0}^{\pi/4} \log(12 \cos^2 \theta - \beta_n^2 \eta^2(\theta)) \, d\theta
	\\
	& \ = \frac{\pi}{4} \log( \frac{\beta_n^2}{2})
	+ \int_{0}^{\pi/4} \log( (2u_{n}-1 - \cos 2\theta) (\cos2\theta +1 - \frac{1}{u_n})) \, d\theta
	\\
	& \ = \frac{\pi}{4} \log( \frac{\beta_n^2}{2})
	+ \frac{1}{2} \int_{0}^{\pi/2} \log(2u_{n}-1 - \cos \theta) \, d\theta
	+ \frac{1}{2} \int_{0}^{\pi/2} \log(\cos\theta +1 - \frac{1}{u_n}) \, d\theta.
\end{align*}
Using this last identity in \eqref{eq:mid} and then using \eqref{eq:Inf1asy}, \eqref{eq:Int1} and \eqref{eq:mid} in \eqref{eq:Inf2asy1},
we obtain \eqref{eq:Inf2asy2}.
\end{proof}
To understand the asymptotic behavior of the integrals $J_{n,1}$ and $J_{n,2}$ in \eqref{eq:Jn1Jn2}, we note that
\[
	\alpha_n
	= \frac{1}{2} + \frac{3}{\beta_n^2}
	= \frac{1}{2} + \frac{12}{\pi^2} \frac{1}{\big( 1+\dfrac{2-n_0}{n}\big)^2}
	= \frac{1}{2} + \frac{12}{\pi^2} - \frac{24}{\pi^2} \dfrac{2-n_0}{n}
	+\mathcal{O} \Big( \frac{1}{n^2} \Big),
\]
\[
	\sqrt{\alpha_n^2-\frac{1}{2}}
	= \frac{\mu}{2\pi^2} - \frac{24}{\pi^2} \frac{24 + \pi^2}{\mu} \dfrac{2-n_0}{n}
	+\mathcal{O} \Big( \frac{1}{n^2} \Big),
\]
so that
\begin{equation} \label{eq:2unm1}
	2u_{n}-1
	= \nu 
	- \frac{48(\nu +1)}{\mu} \dfrac{2-n_0}{n}
	+ \mathcal{O}\Big(\frac{1}{n^{2}} \Big),
	\quad
	\big( \nu \approx 5.6 \big),
\end{equation}
and
\begin{equation} \label{eq:1m1oun}
	1- \frac{1}{u_n}
	= \frac{\nu-1}{\nu+1}
	- \frac{96}{\mu(\nu+1)} \dfrac{2-n_0}{n}
	+ \mathcal{O}\Big(\frac{1}{n^{2}} \Big),
	\quad
	\Big( \frac{\nu-1}{\nu+1} \approx 0.7 \Big).
\end{equation}
We need to study
the following integrals to understand $J_{n,1}$ and $J_{n,2}$:
\[
	J_{1}(\tau_n) = \int_{0}^{\pi/2} \log(\tau_{n}-\cos \theta) \, d\theta
	\quad
	\text{with} \quad \tau_n = 2u_{n}-1 \approx 5.6, \]  
\[ J_{2}(\tau_n) = \int_{0}^{\pi/2} \log(\cos \theta +\tau_{n}) \, d\theta \quad \text{with} \quad	\tau_n = 1- \frac{1}{u_n} \approx 0.7.
\]
For either case we write
\begin{equation} \label{eq:taun}
	\tau_{n} =  a + \frac{b(n)}{n} + \mathcal{O} \Big( \frac{1}{n^{2}} \Big)
\end{equation}
where $b(n)$ is bounded.

\begin{Lem} \label{lemma:J1asymp}
If the sequence $\{ \tau_n\}$ satisfies \eqref{eq:taun} with $a > 1$, then as $n \to \infty$ we have
\[
	J_{1}(\tau_n)
	= 
	J_{11}(a)
	+ J_{12}(a) \frac{b(n)}{n} 
	+ \mathcal{O}\Big(\frac{1}{n^{2}}\Big)
\]
where
\begin{multline*}
	J_{11}(a)
	= 
	\operatorname{Cl}_2(\pi + 2 \arctan(a-\sqrt{a^2-1}))
	- \operatorname{Cl}_2(2 \arctan(a-\sqrt{a^2-1}))
	\\
	- (\frac{\pi}{2}+ 2 \arctan(a-\sqrt{a^2-1})) \, \log (a-\sqrt{a^2-1})
	- \frac{\pi}{2} \log 2
\end{multline*} 
and
\[
	J_{12}(a)
	= \frac{2}{\sqrt{a^2-1}} \arctan \big( \sqrt{\frac{a+1}{a-1}} \big).
\]
\end{Lem}
\begin{proof}
For $a > 1$, we have
\begin{align*}
	J_{1}(\tau_n)
	& = \int_{0}^{\pi/2} \log(a-\cos \theta) \, d\theta
	+ \int_{0}^{\pi/2}\log\Big(1+\frac{\frac{b(n)}{n}+\mathcal{O}(\frac{1}{n^{2}})}{a-\cos\theta}\Big)\, d\theta
	\\
	& = I_1(a) + \frac{b(n)}{n} I_2(a) + \mathcal{O}\Big(\frac{1}{n^{2}}\Big)
\end{align*}
where
\[
	I_1(a) = \int_{0}^{\pi/2} \log(a-\cos \theta) \, d\theta
	\qquad
	\text{and}
	\qquad
	I_2(a) = \int_{0}^{\pi/2} \frac{d\theta}{a-\cos\theta}.
\]
To complete the proof, we show that $I_1(a) = J_{11}(a)$ and $I_2(a) = J_{12}(a)$. For the former,
we make use of \cite[p.308, Formula 39]{L}
\[
	\int_{0}^{\phi} \log (1-2r \, \cos \theta +r^2) d\theta
	= \operatorname{Cl}_2(2\phi + 2 \omega)
	- \operatorname{Cl}_2(2\phi)
	- \operatorname{Cl}_2(2 \omega)
	- 2\omega \, \log r, \; (0<r<1),
\]
where $\displaystyle \omega  = \arg(1-re^{-i\theta})$.
To this end, we set $r = a- \sqrt{a^2-1}$ so that $a= \frac{1+r^2}{2r}$ and therefore
(using $\operatorname{Cl}_2(\pi) = 0$)
\begin{align*}
	I_1(a)
	& = \int_0^{\pi/2} \log \big(\frac{1+r^2}{2r} -\cos \theta \big) \, d\theta
	\\
	& = \int_0^{\pi/2} \log (1 -2r\cos \theta +r^2) \, d\theta
	- \frac{\pi}{2} \log 2r
	= J_{11}(a).
\end{align*}
As for $I_2(a)$, we have
\[
	I_2(a) 
	= \frac{2}{\sqrt{a^2-1}} \arctan \big( \sqrt{\frac{a+1}{a-1}} \tan \frac{\theta}{2} \big) \Big|_{\theta = 0}^{\pi/2}
	=  J_{12}(a).
\]
This finishes the proof.
\end{proof}

\begin{Lem} \label{lemma:J2asymp}
If the sequence $\{ \tau_n\}$ satisfies \eqref{eq:taun} with $a \in (0,1)$, then as $n \to \infty$ we have
\[
	J_{2}(\tau_n)
	= 
	J_{21}(a)
	+ J_{22}(a) \frac{b(n)}{n} 
	+ \mathcal{O}\Big(\frac{1}{n^{2}}\Big)
\]
where
\[
	J_{21}(a)
	= 
	\operatorname{Cl}_2\big( \frac{\pi}{2} + \arccos a \big)
	+ \operatorname{Cl}_2\big( \frac{\pi}{2} - \arccos a \big)
	- \frac{\pi}{2} \log 2
\]
and
\[
	J_{22}(a)
	= \frac{1}{\sqrt{1-a^2}} \log \big( \frac{1+ \sqrt{1-a^2}}{a} \big).
\]
\end{Lem}
\begin{proof}
For $a> 0$, we have
\begin{align*}
	J_{2}(\tau_n) 
	& = \int_{0}^{\pi/2} \log(\cos \theta + a) \, d\theta
	+ \int_{0}^{\pi/2}\log\Big(1+\frac{\frac{b(n)}{n}+\mathcal{O}(\frac{1}{n^{2}})}{\cos\theta+a}\Big)\, d\theta
	\\
	& = I_3(a)
	+ \frac{b(n)}{n} I_4(a)
	+ \mathcal{O}\Big(\frac{1}{n^{2}}\Big)
\end{align*}
where
\[
	I_3(a) = \int_{0}^{\pi/2} \log(\cos \theta + a) \, d\theta
	\qquad
	\text{and}
	\qquad
	I_4(a) =  \int_{0}^{\pi/2} \frac{d\theta}{\cos\theta+a}.
\]
Considering $I_3(a)$, we recall \cite[p.308, Formula 36]{L}
\[
	\int_{0}^{\varphi} \log (1+\sec \phi \, \cos \theta) d\theta
	= \operatorname{Cl}_2(\pi + \phi - \varphi)
	+ \operatorname{Cl}_2(\pi - \phi - \varphi)
	- \varphi \, \log (2 \cos \phi).
\]
Therefore, for $a \in (0,1)$, setting $a = \cos \phi$, we have
\[
	I_3(a)
	= \frac{\pi}{2} \log a
	+ \int_0^{\pi/2} \log(1+ \sec \phi \, \cos \theta) \, d\theta
	= J_{21}(a).
\]
For $I_4(a)$, we have
\[
	I_4(a)
	= \frac{1}{\sqrt{1-a^2}}
	\Big( \log \big(1+ \sqrt{\frac{1-a}{1+a}} \tan \frac{\theta}{2} \big)
	- \log \big(1- \sqrt{\frac{1-a}{1+a}} \tan \frac{\theta}{2} \big) \Big) \Big|_{\theta = 0}^{\pi/2}
	= J_{22}(a).
\]
This completes the proof. 
\end{proof}

\noindent
\emph{Proof of Proposition \ref{prop:Inbetaf2asymp}.}
In light of Lemma \ref{lemma:Inbetaf2asymp1}, equations \eqref{eq:2unm1}, \eqref{eq:1m1oun} and \eqref{eq:taun}, and
Lemmas \ref{lemma:J1asymp} and \ref{lemma:J2asymp}, we have
\begin{align*}
	I_n^{\beta}(f_2)
	& = \frac{2}{\pi} \, n^2 \, \log n
	+ \frac{2}{\pi} \Big( \frac{2G}{\pi} + \log \frac{\sqrt{6}}{\pi} - \frac{1}{\pi} \big( J_{11} (\nu) + J_{21} (\frac{\nu-1}{\nu+1}) \big) \Big) n^2 
	\\
	& + \frac{96}{\pi^2\mu} \Big( (\nu+1) J_{12} (\nu) + \frac{2}{\nu+1} J_{22} (\frac{\nu-1}{\nu+1}) \Big) (2-n_0) n
	+ \mathcal{O}(1).
\end{align*}
Simplifying this expression we obtain Proposition \ref{prop:Inbetaf2asymp}. \hfill $\square$

\section{Asymptotic Behavior of $F_n^{\beta}(f_2)$}

In this section we prove the following for the asymptotic behavior of $F_n^{\beta}(f_2)$.
\begin{Prop} \label{prop:Fnbetaf2}
As $n \to \infty$, we have
\[
	F_n^{\beta}(f_2)
	= \frac{2}{\pi} n^2 \log n +  \alpha(n_0) n^2 + \beta(n_0)n + \mathcal{O} \left( 1 \right)
\]
for some constants $\alpha(n_0)$ and $\beta(n_0)$ where $n_0$ is as defined in \eqref{eq:n0}.
\end{Prop}
The plan of the proof is as follows. First we set
\begin{equation} \label{eq:N}
	N = \frac{n-n_0}{4},
\end{equation}
\begin{equation} \label{eq:AkBkCk}
	A_k = \sqrt{1+4\frac{\pi^2k^2}{3n^2} \Big(1-\frac{\pi^2k^2}{3n^2}\Big)},
	\quad
	B_k = \frac{3n^2}{2\pi^2} (1+A_k),
	\quad
	C_k 
	=  \frac{3n^2}{2\pi^2} (A_k-1),
\end{equation}
and, in \S\S4.1, we utilize a partial fraction decomposition along with the functional relations and the asymptotic behavior of the digamma function
\cite{Z} to show the following.
\begin{Lem} \label{prop:decompFnbetaf2}
As $n \to \infty$, we have
\begin{equation} \label{eq:slFn2enew}
	F_n^{\beta}(f_2)
	= \dfrac{2n^2}{\pi^2} \left( R_{n,1} -2 R_{n,2} + R_{n,3} + \pi R_{n,4} + 2\pi R_{n,5} + Q_n \right) + \mathcal{O}(1)
\end{equation}
where
\begin{equation} \label{eq:Rn12}
	R_{n,1}
	= \sum_{k=1}^{N}
	\frac{1}{N} \frac{1}{A_k} \frac{N}{\sqrt{B_k}}
	\log \Big( \frac{1+\frac{N}{\sqrt{B_k}}}{1-\frac{N}{\sqrt{B_k}}} \Big),
	\quad
	R_{n,2}
	= \sum_{k=1}^{N}
	\frac{1}{N} \frac{1}{A_k} \frac{\arctan\big(\frac{\sqrt{C_k}}{N}\big)}{\frac{\sqrt{C_k}}{N}},
\end{equation}
\begin{equation} \label{eq:Rn34}
	R_{n,3}
	= \sum_{k=1}^{N}
	\frac{1}{k^2+N^2-\frac{\pi^2}{3n^2}(k^4+N^4)},
	\quad
	R_{n,4}
	= \sum_{k=1}^N \frac{1}{A_k\sqrt{C_k}},
\end{equation}
\begin{equation} \label{eq:Rn5Qn}
	R_{n,5}
	= \sum_{k=1}^N \frac{1}{A_k\sqrt{C_k}} \, \frac{e^{-2\pi \sqrt{C_k}}}{1-e^{-2\pi \sqrt{C_k}}},
	\quad
	Q_n = \sum_{k = 1}^{N} \dfrac{1}{k^2-\frac{\pi^2}{3n^2}k^4}.
\end{equation}
\end{Lem}
In \S\S4.2 we estimate the summands in $R_{n,j}$. In \S\S4.3-4.7 we use these estimates to study the asymptotic behavior of $R_{n,j}$ and $Q_n$
and prove that
\begin{Lem} \label{lemma:asymptotics}
As $n \to \infty$, we have
\begin{equation} \label{eq:Rn12asymp}
	R_{n,\ell}
	= \alpha_{\ell}
	+ \beta_{\ell}(n_0) \frac{1}{n}
	+ \gamma_{\ell}(n_0) \frac{1}{n^2}
	+ \mathcal{O} \Big( \frac{1}{n^3} \Big),
	\qquad
	\ell =1,2,
\end{equation}
\begin{align} \label{eq:Rn3asymp}
	R_{n,3} = \beta_{3} \frac{1}{n} + \mathcal{O} \Big( \frac{1}{n^2} \Big),
\end{align}
\begin{equation} \label{eq:Rn4asymp}
	R_{n,4}
	= \log n
	+ \alpha_4(n_0)
	+ \beta_4(n_0) \frac{1}{n}
	+ \gamma_4(n_0) \frac{1}{n^2}
	+ \mathcal{O} \Big( \frac{\log n}{n^3} \Big),
\end{equation}
\begin{equation} \label{eq:Rn5asymp}
	R_{n,5}
	= \log \Big( \frac{2\pi^\frac{3}{4}}{e^{\frac{\pi}{12}} \Gamma(\frac{1}{4})}\Big)
	+ \mathcal{O} \Big( \frac{1}{n^2} \Big),
\end{equation}
\begin{equation} \label{eq:Qnasymp}
	Q_{n} = \frac{\pi^2}{6}
	+ \Big( \frac{\pi}{2\sqrt{3}} \log\Big( \frac{4\sqrt{3}+\pi}{4\sqrt{3}-\pi}\Big) - 4 \Big) \frac{1}{n}
	+ \mathcal{O} \Big( \frac{1}{n^2}\Big),
\end{equation}
for some constants $\alpha_{\ell}$, $\beta_{\ell}$, and $\gamma_{\ell}$ some of which depend on $n_0$.
\end{Lem}
The use of \eqref{eq:Rn12asymp}--\eqref{eq:Qnasymp} in \eqref{eq:slFn2enew} proves Proposition \ref{prop:Fnbetaf2}.
In principle, the constants in \eqref{eq:Rn12asymp} and \eqref{eq:Rn4asymp} can be computed explicitly. As will be apparent
from the derivations that follow, however, this would require the evaluation of some nontrivial integrals. 
The constant $\beta_3$ in \eqref{eq:Rn3asymp} is given explicitly at the end of \S\S4.4.

Concerning the proof of Lemma \ref{lemma:asymptotics}, let us mention that the study of the sums $R_{n,1}$ and $R_{n,2}$ in \S\S4.3
leading into the asymptotic relation \eqref{eq:Rn12asymp} is based on further decompositions
of these sums and utilization of the following version of the Euler-Maclaurin formula due to Lampret \cite{La}.
\vskip.2cm
\noindent
{\bf Theorem C.}\;\;
\emph{For any $N,p \in \mathbb{N}$ and any function $g\in C^{p}[0,1]$,
\begin{align*}
	\sum_{k=1}^{N} \frac{1}{N}  g \big( \frac{k}{N} \big) 
	= \int_{0}^{1}g(x)\, dx + \frac{1}{N} (g(1)-g(0))
	+ \sum_{\ell=1}^{p} \frac{1}{N^{\ell}} \frac{B_{\ell}}{{\ell}!} \big[g^{({\ell}-1)}(x)\big]_{0}^{1} + r_{p,N}(g),
\end{align*}
where $B_{\ell}$ are the Bernoulli numbers $(B_1 = -\frac{1}{2}$, $B_2 = \frac{1}{6}$, $B_3=0$, $B_4=-\frac{1}{30}, \ldots)$, and
\[
	r_{p,N}(g) = -\frac{1}{N^p}\frac{1}{p!} \int_{0}^{1} B_{p} (-Nx) \, g^{(p)}(x)\, dx
\]
with $B_{p}(x)$ being the $p$-th Bernoulli polynomial in $[0,1)$ extended to all  $x \in \mathbb{R}$ via
$B_{p}(x+1)=B_{p}(x)$.} \vspace{0.1cm}

\noindent
The study of $R_{n,4}$ in \S\S4.5 resulting in the asymptotic relation \eqref{eq:Rn4asymp} is similar but additionally involves a
singularity extraction technique. As for $R_{n,5}$, a different and more delicate analytical approach is utilized in \S\S4.6 to prove
the relation \eqref{eq:Rn5asymp}. Concerning $R_{n,3}$ and $Q_n$, in \S\S4.4 and \S\S4.7, we employ partial fraction decompositions
along with the functional
relations and the asymptotic behavior of the digamma function and prove relations \eqref{eq:Rn3asymp} and \eqref{eq:Qnasymp}.

\subsection{Proof of Lemma \ref{prop:decompFnbetaf2}}
With $N$ as defined in \eqref{eq:N}, we see from \eqref{eq:Dneps} that $\mathbf{t}_{j,k} \in D_n^{\beta}$
if and only if $ |j|,|k| \le N$ and $(j,k) \ne (0,0)$ (see \eqref{eq:explainthis}).
Accordingly
\begin{align} \label{eq:slFn2e}
	F_n^{\beta} (f_2)
	& = \sum_{\mathbf{t}_{j,k} \in D_n^{\beta}} f_2(\mathbf{t}_{j,k}) 
	= \dfrac{n^2}{\pi^2} \sum_{|j|,|k| \le N \atop (j,k) \ne (0,0)} \dfrac{1}{j^2+k^2-\frac{\pi^2}{3n^2}(j^4+k^4)}
	\\
	& = \dfrac{4n^2}{\pi^2}
	\Big(
		\sum_{k = 1}^{N} \dfrac{1}{k^2-\frac{\pi^2}{3n^2}k^4}
		+ \sum_{j,k = 1}^{N} \dfrac{1}{j^2+k^2-\frac{\pi^2}{3n^2}(j^4+k^4)}
	\Big)
	\nonumber
	\\
	& = \dfrac{4n^2}{\pi^2} \left( Q_n + R_n \right),
	\quad \text{say (see \eqref{eq:Rn5Qn})}.
	\nonumber	
\end{align}
Considering the sum $R_n$, we use the partial fraction decomposition
\begin{equation} \label{eq:parfrac}
	\frac{1}{x^2-ax^4+b}
	= \frac{1}{2A\sqrt{B}} \big( \frac{1}{x+\sqrt{B}} - \frac{1}{x-\sqrt{B}} \big)
	+ \frac{i}{2A\sqrt{C}} \big( \frac{1}{x+i\sqrt{C}} - \frac{1}{x-i\sqrt{C}} \big)
\end{equation}
(valid for $a,b>0$ with $A = \sqrt{1+4ab}$, $B = \dfrac{1+A}{2a}$, and $C = \dfrac{A-1}{2a} = \dfrac{2b}{1+A}$)
for the parameters $x=j$, $a = \dfrac{\pi^2}{3n^2}$, $b=k^2-\dfrac{\pi^2}{3n^2}k^4$,
and the relation \cite[Formula 6.3.6]{AS}
\begin{equation} \label{eq:sum1}
	\sum_{j=1}^{N} \frac{1}{j+a} = \psi(N+1+a)-\psi(1+a),
	\qquad
	(a \notin -\mathbb{N}),
\end{equation}
for the digamma function to evaluate the sum with respect to $j$ to have
\begin{align*}
	R_n
	& = \sum_{k=1}^{N}
	\frac{\psi(N+1+\sqrt{B_k}) - \psi(1+\sqrt{B_k}) - \psi(N+1-\sqrt{B_k}) +  \psi(1-\sqrt{B_k})}{2A_k\sqrt{B_k}}
	\\
	& + i \sum_{k=1}^{N}
	\frac{\psi(N+1+i\sqrt{C_k}) - \psi(1+i\sqrt{C_k}) - \psi(N+1-i\sqrt{C_k}) + \psi(1-i\sqrt{C_k})}{2A_k\sqrt{C_k}}
\end{align*}
where $A_k$, $B_k$, and $C_k$ are as defined in \eqref{eq:AkBkCk}.
Next we use the functional relations \cite[Formulas 6.3.5 \& 6.3.7]{AS}
\begin{equation} \label{eq:formula}
	\psi(1+z) = \psi(z) + \frac{1}{z}
	\quad
	\text{and}
	\quad
	\psi(1-z) = \psi(z) + \pi \cot \pi z
\end{equation}
for the digamma function to write
\begin{align} \label{eq:relations}
	& \psi(N+1+\sqrt{B_k}) = \psi(\sqrt{B_k}+N) + \frac{1}{\sqrt{B_k}+N},
	\\
	& \psi(N+1-\sqrt{B_k}) = \psi(\sqrt{B_k}-N) + \pi \cot \pi (\sqrt{B_k}-N),
	\nonumber
	\\
	& \psi(1+\sqrt{B_k}) - \psi(1-\sqrt{B_k}) = \frac{1}{\sqrt{B_k}} - \pi \cot \pi \sqrt{B_k},
	\nonumber
	\\
	& \psi(N+1+i\sqrt{C_k}) = \psi(N+i\sqrt{C_k}) + \frac{1}{N+i\sqrt{C_k}},
	\nonumber
	\\
	& \psi(N+1-i\sqrt{C_k}) = \psi(N-i\sqrt{C_k}) + \frac{1}{N-i\sqrt{C_k}},
	\nonumber
	\\
	& \psi(1+i\sqrt{C_k}) - \psi(1-i\sqrt{C_k}) = \frac{1}{i\sqrt{C_k}} - \pi \cot \pi i \sqrt{C_k},
	\nonumber
\end{align}
and use the fact that the cotangent function is $\pi$-periodic to have
\begin{align} \label{eq:useagain}
	R_n
	& = \sum_{k=1}^{N}
	\frac{1}{2A_k\sqrt{B_k}} 
	\Big(
		\psi(\sqrt{B_k}+N) - \psi(\sqrt{B_k}-N) + \frac{1}{\sqrt{B_k}+N} - \frac{1}{\sqrt{B_k}}
	\Big)
	\\
	& + i \sum_{k=1}^{N}
	\frac{1}{2A_k\sqrt{C_k}} 
	\Big(
		\psi(N+i\sqrt{C_k}) - \psi(N-i\sqrt{C_k})
	\nonumber	
	\\	
	& \hspace{5cm} - \frac{2i\sqrt{C_k}}{C_k+N^2} - \frac{1}{i\sqrt{C_k}} + \pi \cot \pi i \sqrt{C_k}
	\Big).
	\nonumber
\end{align}
For $1 \le k \le N$, the estimates
\begin{equation} \label{eq:ABCestimate}
	1 < A_k < \frac{13}{10},
	\
	\frac{3 n^2}{\pi^2}  < B_k < \frac{7 n^2}{2\pi^2} ,
	\
	\frac{69 k^2}{100}  < C_k < k^2,
	\
	\frac{3n}{10} < \sqrt{B_k} \pm N < \frac{17n}{20}
\end{equation}
are easily shown to hold, and they allow us to employ the asymptotic expansion
\cite[Formula 6.3.18]{AS}
\begin{equation} \label{eq:asymp}
	\psi(z) = \log z - \frac{1}{2z} - \frac{1}{12z^2} + \mathcal{O} \Big( \frac{1}{|z|^4}\Big),
	\quad
	\text{as } |z| \to \infty \text{ with } |\operatorname{arg} z| < \pi - \delta,
\end{equation}
in \eqref{eq:useagain} to deduce
\begin{multline*}
	R_n
	= \sum_{k=1}^{N}
	\frac{1}{2A_k\sqrt{B_k}} 
	\Big(
		\log\big(\frac{\sqrt{B_k}+N}{\sqrt{B_k}-N}\big) + \frac{\sqrt{B_k}}{B_k-N^2} - \frac{1}{\sqrt{B_k}}
	\Big)
	\\
	+ i \sum_{k=1}^{N}
	\frac{1}{2A_k\sqrt{C_k}} 
	\Big(
		2i \arctan\frac{\sqrt{C_k}}{N} - \frac{i\sqrt{C_k}}{C_k+N^2} - \frac{1}{i\sqrt{C_k}} + \pi \cot \pi i \sqrt{C_k}  
	\Big) \!
	+ \! \mathcal{O}\Big( \frac{1}{n^2} \Big)
\end{multline*}
which we rewrite as
\begin{align} \label{eq:Hndecomppre}
	R_n
	& = \dfrac{1}{2} \sum_{k=1}^{N}
	\frac{1}{N} \frac{1}{A_k} \frac{N}{\sqrt{B_k}}
	\log \Big( \frac{1+\frac{N}{\sqrt{B_k}}}{1-\frac{N}{\sqrt{B_k}}} \Big)
	- \sum_{k=1}^{N}
	\frac{1}{N} \frac{1}{A_k} \frac{\arctan\big(\frac{\sqrt{C_k}}{N}\big)}{\frac{\sqrt{C_k}}{N}} 
	\\
	& + \dfrac{1}{2} \sum_{k=1}^{N}
	\frac{1}{A_k} 
	\Big( \frac{1}{B_k-N^2} + \frac{1}{C_k+N^2} \Big)
	\nonumber
	- \dfrac{1}{2} \sum_{k=1}^{N} \frac{1}{A_k} \Big( \frac{1}{B_k} + \frac{1}{C_k} \Big)
	\\
	& + \dfrac{1}{2} \sum_{k=1}^{N} \frac{\pi i \cot \pi i \sqrt{C_k}}{A_k\sqrt{C_k}}
	+ \mathcal{O}\Big( \frac{1}{n^2} \Big).
	\nonumber
\end{align}
Next, in \eqref{eq:Hndecomppre}, we use \eqref{eq:AkBkCk} in the third and fourth sums and the definition
of the cotangent function in the last sum to obtain
\begin{align} \label{eq:Hndecomp}
	R_n
	& = \dfrac{1}{2} \sum_{k=1}^{N}
	\frac{1}{N} \frac{1}{A_k} \frac{N}{\sqrt{B_k}}
	\log \Big( \frac{1+\frac{N}{\sqrt{B_k}}}{1-\frac{N}{\sqrt{B_k}}} \Big)
	- \sum_{k=1}^{N}
	\frac{1}{N} \frac{1}{A_k} \frac{\arctan\big(\frac{\sqrt{C_k}}{N}\big)}{\frac{\sqrt{C_k}}{N}} 
	\nonumber
	\\
	& + \dfrac{1}{2} \sum_{k=1}^{N}
	\frac{1}{k^2+N^2-\frac{\pi^2}{3n^2}(k^4+N^4)}
	- \dfrac{1}{2} Q_n
	\nonumber
	\\
	& + \frac{\pi}{2} \sum_{k=1}^N \frac{1}{A_k\sqrt{C_k}} 
	+ \pi \sum_{k=1}^N \frac{1}{A_k\sqrt{C_k}} \, \frac{e^{-2\pi \sqrt{C_k}}}{1-e^{-2\pi \sqrt{C_k}}}
	+ \mathcal{O}\Big( \frac{1}{n^2} \Big)
	\nonumber
	\\
	& = \dfrac{1}{2} R_{n,1} - R_{n,2} + \dfrac{1}{2} R_{n,3} -\frac{1}{2} Q_{n} + \dfrac{\pi}{2} R_{n,4} + \pi R_{n,5}
	+ \mathcal{O}\Big( \frac{1}{n^2} \Big)
\end{align}
(see \eqref{eq:Rn12}--\eqref{eq:Rn5Qn}). Finally, we use \eqref{eq:Hndecomp} in \eqref{eq:slFn2e} to get \eqref{eq:slFn2enew},
and this completes the proof of Lemma \ref{prop:decompFnbetaf2}.

\subsection{Approximation of the summands}
Here we study the asymptotic behavior of the summands in $R_{n,j}$ for $j=1,2,4$. To this end, we introduce the notation
\begin{equation} \label{eq:ak}
	a_k = \frac{\pi}{4\sqrt{3}} \frac{k}{N},
\end{equation}
\begin{equation} \label{eq:1overN0}
	\frac{1}{N_0} =
	\left\{
		\begin{array}{cl}
			\frac{n_0}{4}\frac{1}{N}, & n_0 \in \{1,2,3\}, \\
			0, & n_0 = 0,
		\end{array}
	\right.
\end{equation}
and use \eqref{eq:N} in \eqref{eq:AkBkCk} to write
\[
	A_k 
	= \sqrt{1+ \frac{4a_k^2}{(1+\frac{1}{N_0})^2}
	\Big( 1-\frac{a_k^2}{(1+ \frac{1}{N_0})^2} \Big)}
\]
(note that $0 <a_k < \frac{1}{2}$). For a constant $0<a<\frac{1}{2}$,
\[
	\Psi(x) = \sqrt{1+\frac{4a^2}{(1+x)^2} \Big( 1-\frac{a^2}{(1+x)^2} \Big)}
\]
is smooth on the interval $(-\frac{1}{2},\infty)$. 
We set $$A = \sqrt{1+4a^2(1-a^2)}.$$
By Taylor's theorem with remainder, we have
\begin{align*}
	\frac{1}{\Psi(x)}
	& = \frac{1}{A} \Big\{ 1- \frac{4a^2(2a^2-1)}{A^2} \, x + \frac{2a^2(8a^6+4a^4+10a^2-3)}{A^4} \, x^2 + \mathcal{O}(x^3) \Big\},
	\\
	\sqrt{1+\Psi(x)}
	& = \sqrt{1+A}
	\Big\{ 1 + \frac{2a^2(2a^2-1)}{A(1+A)} \, x
	\\
	& \hspace{0.4cm}
	+ \frac{a^2}{A^2(1+A)} \Big( \frac{(2a^2-3)(12a^4-1)}{A} - \frac{2a^2(2a^2-1)^2}{1+A} \Big) \, x^2 + \mathcal{O}(x^3) \Big\},
	\\
	\frac{1}{\sqrt{1+\Psi(x)}}
	& = \frac{1}{\sqrt{1+A}}
	\Big\{ 1 - \frac{2a^2(2a^2-1)}{A(1+A)} \, x
	\\
	& \hspace{0.37cm}
	+ \frac{a^2}{A^2(1+A)} \Big( \frac{6a^2(2a^2-1)^2}{1+A} - \frac{(2a^2-3)(12a^4-1)}{A} \Big) \, x^2 + \mathcal{O}(x^3) \Big\}
\end{align*}
which are valid on $(- \frac{1}{2},\infty)$. Setting
\[
	\mathcal{A}_k
	= \sqrt{1+4 a_k^2 \left( 1- a_k^2 \right)},
\]
and implementing these Taylor approximations for $a = a_k$ and $x = \frac{1}{N_0}$, we obtain
\begin{equation} \label{eq:1overAk}
	\frac{1}{A_k}
	= \alpha_{k,1}
	\Big\{ 1 + \beta_{k,1} \frac{1}{N_0} + \gamma_{k,1} \frac{1}{N_0^2} + \mathcal{O}\Big( \frac{1}{N_0^3} \Big) \Big\}
\end{equation}
with
\[
	\alpha_{k,1} = \frac{1}{\mathcal{A}_k},
	\quad
	\beta_{k,1} = - \frac{4a_k^2(2a_k^2-1)}{\mathcal{A}_k^2},
	\quad
	\gamma_{k,1} = \frac{2a_k^2(8a_k^6+4a_k^4+10a_k^2-3)}{\mathcal{A}_k^4},
\]
\begin{align*}
	\sqrt{1 + A_k}
	& = \alpha_{k,2}
	\Big\{ 1 + \beta_{k,2} \frac{1}{N_0} + \gamma_{k,2} \frac{1}{N_0^2} + \mathcal{O}\Big( \frac{1}{N_0^3} \Big) \Big\}
\end{align*}
with
\[
	\alpha_{k,2} = \sqrt{1 + \mathcal{A}_k},
	\quad
	\beta_{k,2} = \frac{2a_k^2(2a_k^2-1)}{\mathcal{A}_k(1+\mathcal{A}_k)},
\]
\[
	\gamma_{k,2} = \frac{a_k^2}{\mathcal{A}_k^2(1+\mathcal{A}_k)}
	\Big( \frac{(2a_k^2-3)(12a_k^4-1)}{\mathcal{A}_k} - \frac{2a_k^2(2a_k^2-1)^2}{1+\mathcal{A}_k} \Big),
\]
and
\begin{align} \label{eq:1pAsqrtres}
	\frac{1}{\sqrt{1 + A_k}}
	& = \alpha_{k,3}
	\Big\{ 1 + \beta_{k,3} \frac{1}{N_0} + \gamma_{k,3} \frac{1}{N_0^2} + \mathcal{O}\Big( \frac{1}{N_0^3} \Big) \Big\}
\end{align}
with
\[
	\alpha_{k,3} = \frac{1}{\sqrt{1 + \mathcal{A}_k}},
	\quad
	\beta_{k,3} = - \beta_{k,2},
\]
\[
	\gamma_{k,3} = \frac{a_k^2}{\mathcal{A}_k^2(1+\mathcal{A}_k)}
	\Big( \frac{6a_k^2(2a_k^2-1)^2}{1+\mathcal{A}_k} - \frac{(2a_k^2-3)(12a_k^4-1)}{\mathcal{A}_k} \Big).
\]
Since 
\[
	\frac{1}{1 + \frac{1}{N_0}} = 1 - \frac{1}{N_0} + \frac{1}{N_0^2} + \mathcal{O}\Big( \frac{1}{N_0^3} \Big),
\]
\eqref{eq:1pAsqrtres} entails
\begin{align} \label{eq:NoversqrtBk}
	\frac{N}{\sqrt{B_k}}
	& = \frac{2\pi}{\sqrt{6}} \frac{1}{\sqrt{1+A_k}} \frac{N}{n} 
	= \frac{\pi}{2\sqrt{6}} \frac{1}{\sqrt{1+A_k}} \frac{1}{1 + \frac{1}{N_0}} 
	\\
	& = \alpha_{k,4}
	\Big\{
		1
		+ \beta_{k,4} \frac{1}{N_0}
		+ \gamma_{k,4} \frac{1}{N_0^2}
		+ \mathcal{O}\Big( \frac{1}{N_0^3} \Big)
	\Big\}
	\nonumber
\end{align}
with
\[
	\alpha_{k,4} = \frac{\pi}{2\sqrt{6}} \alpha_{k,3},
	\quad
	\beta_{k,4} = \beta_{k,3} -1,
	\quad
	\gamma_{k,4} = 1 - \beta_{k,3} + \gamma_{k,3},
\]
so that
\[
	1\pm \frac{N}{\sqrt{B_k}}
	= (1\pm \alpha_{k,4})
	\Big\{
		1
		\pm \frac{\alpha_{k,4}}{1\pm \alpha_{k,4}}
		\Big( \beta_{k,4} \frac{1}{N_0} + \gamma_{k,4} \frac{1}{N_0^2} + \mathcal{O}\Big( \frac{1}{N_0^3} \Big) \Big)
	\Big\}.
\]
Therefore
\begin{align} \label{eq:prelog}
	\log \Big( 1 \pm \frac{N}{\sqrt{B_k}} \Big)
	& = \log (1\pm \alpha_{k,4})
	\pm \frac{\alpha_{k,4}\beta_{k,4}}{1 \pm \alpha_{k,4}} \frac{1}{N_0}
	\\
	& + \Big(
		\pm \frac{\alpha_{k,4}\gamma_{k,4}}{1 \pm \alpha_{k,4}}
		- \frac{1}{2} \Big( \frac{\alpha_{k,4}\beta_{k,4}}{1 \pm \alpha_{k,4}} \Big)^2
	\Big)
	\frac{1}{N_0^2} 
	+ \mathcal{O}\Big( \frac{1}{N_0^3} \Big).
	\nonumber
\end{align}
Accordingly, \eqref{eq:prelog} implies
\begin{align} \label{eq:log}
	\log \Big( \frac{1+\frac{N}{\sqrt{B_k}}}{1-\frac{N}{\sqrt{B_k}}} \Big)
	= \log(\alpha_{k,5})
	+ \beta_{k,5} \frac{1}{N_0} + \gamma_{k,5} \frac{1}{N_0^2} + \mathcal{O}\Big( \frac{1}{N_0^3} \Big)
\end{align}
with
\[
	\alpha_{k,5} = \frac{1+\alpha_{k,4}}{1-\alpha_{k,4}},
	\quad
	\beta_{k,5} = \frac{2\alpha_{k,4}\beta_{k,4}}{1-\alpha_{k,4}^2},
	\quad
	\gamma_{k,5}
	= \frac{2\alpha_{k,4}\gamma_{k,4}}{1-\alpha_{k,4}^2} + \frac{2\alpha_{k,4}^3\beta_{k,4}^2}{(1-\alpha_{k,4}^2)^2}.
\]
Next, using \eqref{eq:AkBkCk}, we have
\[
	C_k 
	= \frac{3n^2}{2\pi^2} (A_k-1)
	= \frac{3n^2}{2\pi^2} \frac{A_k^2-1}{1+A_k}
	= \frac{2}{1+A_k} k^2 \Big( 1 - \frac{\pi^2k^2}{3n^2} \Big)
\]
so that, by \eqref{eq:ak} and \eqref{eq:1overN0},
\begin{equation} \label{eq:SqrtCoNpre}
	\frac{\sqrt{C_k}}{N}
	= \frac{\sqrt{2}}{\sqrt{1+A_k}} \frac{k}{N} \sqrt{1 - \frac{\pi^2 k^2}{3n^2}}
	= \frac{\sqrt{2}}{\sqrt{1+A_k}} \frac{k}{N} \sqrt{1 - \frac{a_k^2}{(1+ \frac{1}{N_0})^2}}.
\end{equation}
In \eqref{eq:SqrtCoNpre}, using \eqref{eq:1pAsqrtres} along with the Taylor approximation
\[
	\sqrt{1 - \frac{a^2}{(1+x)^2}}
	= \sqrt{1-a^2}
		\Big\{ 1 + \frac{a^2}{1-a^2} \, x + \frac{a^2(2a^2-3)}{2(1-a^2)^2} \, x^2 + \mathcal{O}(x^3)
	\Big\}
\]
(which is valid on $(-\frac{1}{2},\infty)$ when $0 < a < \frac{1}{2}$),
we deduce
\begin{equation} \label{eq:SqrtCoN}
	\frac{\sqrt{C_k}}{N}
	= \alpha_{k,6}
	\Big\{
		1
		+ \beta_{k,6} \frac{1}{N_0}
		+ \gamma_{k,6} \frac{1}{N_0^2}
		+ \mathcal{O}\Big( \frac{1}{N_0^3} \Big)
	\Big\}
\end{equation}
with
\[
	\alpha_{k,6} = \frac{\sqrt{2}\sqrt{1-a_k^2}}{\sqrt{1+\mathcal{A}_k}} \frac{k}{N},
	\
	\beta_{k,6} = \beta_{k,3} + \frac{a_k^2}{1-a_k^2},
	\
	\gamma_{k,6}
	= \gamma_{k,3} + \frac{\beta_{k,3}a_k^2}{1-a_k^2}
	+\frac{a_k^2(2a_k^2-3)}{2(1-a_k^2)^2}.
\]
In light of \eqref{eq:SqrtCoN}, using the Taylor approximation ($\alpha \ne 0$, $x>-1$)
\begin{align*}
	\frac{\arctan (\alpha(1+x))}{\alpha(1+x)}
	& = \frac{\arctan \alpha}{\alpha} \left( 1 - x + x^2 \right)
	+ \frac{1}{1+\alpha^2} x - \frac{1+2\alpha^2}{(1+\alpha^2)^2} x^2 + \mathcal{O}(x^3),
\end{align*}
we therefore get
\begin{equation} \label{eq:arctan}
	\frac{\arctan\big(\frac{\sqrt{C_k}}{N}\big)}{\frac{\sqrt{C_k}}{N}}
	= \alpha_{k,7}
	\Big\{
		1
		+ \beta_{k,7} \frac{1}{N_0}
		+ \gamma_{k,7}\frac{1}{N_0^2}
	\Big\}
	+ \beta_{k,7}' \frac{1}{N_0}
	+ \gamma_{k,7}' \frac{1}{N_0^2}
	+ \mathcal{O} \Big( \frac{1}{N_0^3} \Big)
\end{equation}
with
\[
	\alpha_{k,7} = \frac{\arctan \alpha_{k,6}}{\alpha_{k,6}},
	\quad
	\beta_{k,7} = - \beta_{k,6},
	\quad
	\gamma_{k,7} = \beta_{k,6}^2-\gamma_{k,6},
\]
and
\[
	\beta_{k,7}' = \frac{\beta_{k,6}}{1+\alpha_{k,6}^2},
	\quad
	\gamma_{k,7}' = \frac{\gamma_{k,6}}{1+\alpha_{k,6}^2} - \frac{\beta_{k,6}^2(1+2\alpha_{k,6}^2)}{(1+\alpha_{k,6}^2)^2}.
\]
Concerning the summand in $R_{n,1}$, we combine \eqref{eq:1overAk}, \eqref{eq:NoversqrtBk} and \eqref{eq:log} to get
\begin{equation} \label{eq:Hn1pre}
	\frac{1}{A_k} \frac{N}{\sqrt{B_k}} 
	\log \left(\frac{1+\frac{N}{\sqrt{B_k}}}{1-\frac{N}{\sqrt{B_k}}} \right)
	= \alpha_{k,8}
	+ \beta_{k,8} \frac{1}{N_0} 
	+ \gamma_{k,8} \frac{1}{N_0^2} 
	+ \mathcal{O} \Big( \frac{1}{N_0^3} \Big)
\end{equation}
with
\[
	\alpha_{k,8} = \alpha_{k,1} \alpha_{k,4} \log(\alpha_{k,5}),
	\quad
	\beta_{k,8} = \alpha_{k,1} \alpha_{k,4} ( (\beta_{k,1}+\beta_{k,4}) \log(\alpha_{k,5}) + \beta_{k,5}),
\]
\[
	\gamma_{k,8}
	= \alpha_{k,1} \alpha_{k,4}
	\big((\beta_{k,1}\beta_{k,4}+\gamma_{k,1}+\gamma_{k,4}) \log(\alpha_{k,5}) + (\beta_{k,1}+\beta_{k,4}) \beta_{k,5} + \gamma_{k,5}\big).
\]
For the summand in $R_{n,2}$, we use \eqref{eq:1overAk} and \eqref{eq:arctan} to obtain
\begin{equation} \label{eq:Hn2pre}
	\frac{1}{A_k} \frac{\arctan\big(\frac{\sqrt{C_k}}{N}\big)}{\frac{\sqrt{C_k}}{N}}
	= \alpha_{k,9}
	+ \beta_{k,9}\frac{1}{N_0} 
	+ \gamma_{k,9} \frac{1}{N_0^2} 
	+ \mathcal{O} \Big( \frac{1}{N_0^3} \Big)
\end{equation}
with 
\[
	\alpha_{k,9} = \alpha_{k,1} \alpha_{k,7},
	\quad
	\beta_{k,9} = \alpha_{k,1} ( \alpha_{k,7} (\beta_{k,1} + \beta_{k,7}) + \beta_{k,7}'),
\]
\[
	\gamma_{k,9}
	= \alpha_{k,1} ( (\beta_{k,1} \beta_{k,7} + \gamma_{k,1} + \gamma_{k,7}) \alpha_{k,7} + \beta_{k,1}\beta_{k,7}' + \gamma_{k,7}').
\] 
As for the term $\dfrac{1}{A_k\sqrt{C_k}}$ appearing in $R_{n,4}$ and $R_{n,5}$, we use \eqref{eq:SqrtCoN} to deduce
\begin{equation} \label{eq:1osqrtC}
	\frac{1}{\sqrt{C_k}}
	= \frac{1}{k} \, \alpha_{k,10}
	\Big\{ 1 + \beta_{k,10} \frac{1}{N_0} + \gamma_{k,10} \frac{1}{N_0^2} + \mathcal{O} \Big( \frac{1}{N_0^3} \Big) \Big\}
\end{equation}
with
\[
	\alpha_{k,10} = \frac{\sqrt{1+\mathcal{A}_k}}{\sqrt{2}\sqrt{1-a_k^2}},
	\quad
	\beta_{k,10} = - \beta_{k,6},
	\quad
	\gamma_{k,10} = \beta_{k,6}^2-\gamma_{k,6},
\]
and we combine \eqref{eq:1overAk} with \eqref{eq:1osqrtC} to obtain
\begin{align} \label{eq:Hndecomppart}
	\frac{1}{A_k\sqrt{C_k}}
	& = \frac{1}{k} \, \alpha_{k,11}
	\Big\{ 1 + \beta_{k,11} \frac{1}{N_0} + \gamma_{k,11} \frac{1}{N_0^2} + \mathcal{O} \Big( \frac{1}{N_0^3} \Big) \Big\}
\end{align}
with
\[
	\alpha_{k,11} = \alpha_{k,1} \alpha_{k,10},
	\quad
	\beta_{k,11} = \beta_{k,1} + \beta_{k,10},
	\quad
	\gamma_{k,11} = \beta_{k,1} \beta_{k,10} + \gamma_{k,1} + \gamma_{k,10}.
\]

\subsection{Asymptotic behavior of $R_{n,1}$ and $R_{n,2}$}
Concerning the asymptotic behavior of $R_{n,1}$ and $R_{n,2}$, here we prove \eqref{eq:Rn12asymp}. In view of \eqref{eq:Hndecomp},
we see that \eqref{eq:Hn1pre} and \eqref{eq:Hn2pre} imply for $\ell = 1,2$
\begin{align} 
	R_{n,\ell}
	& = \sum_{k=1}^{N} \frac{1}{N} \mu_{\ell,1}(\frac{k}{N})
	+ \frac{1}{N_0} \sum_{k=1}^{N} \frac{1}{N} \mu_{\ell,2}(\frac{k}{N})
	+ \frac{1}{N_0^2} \sum_{k=1}^{N} \frac{1}{N} \mu_{\ell,3}(\frac{k}{N})
	+ \mathcal{O} \Big( \frac{1}{N_0^3} \Big)
	\nonumber
	\\
	& = R_{n,\ell,1} + \frac{1}{N_0} R_{n,\ell,2} + \frac{1}{N_0^2} R_{n,\ell,3} + \mathcal{O} \Big( \frac{1}{N_0^3} \Big),
	\quad
	\text{say},
	\label{eq:Hnell}
\end{align}
where the functions $\mu_{1,1}(x)$, $\mu_{1,2}(x)$, $\mu_{1,3}(x)$ are obtained by replacing $\displaystyle \frac{k}{N}$ with $x$ in $\alpha_{k,8}$,
$\beta_{k,8}$, and $\gamma_{k,8}$ respectively, and $\mu_{2,1}(x)$, $\mu_{2,2}(x)$, $\mu_{2,3}(x)$ are similarly
obtained from $\alpha_{k,9}$, $\beta_{k,9}$, and $\gamma_{k,9}$ respectively. The functions $\mu_{\ell,j}(x)$ ($1 \le \ell \le 2$ and $1 \le j \le 3$)
so obtained are smooth in the interval $[0,1]$, and therefore Theorem C applies
to each of the sums $R_{n,\ell,j}$ ($1 \le \ell \le 2$ and $1 \le j \le 3$) to yield
\begin{equation} \label{eq:Hnjlasymp}
	R_{n,\ell,j} = a_{\ell,j} + b_{\ell,j} \frac{1}{N} + c_{\ell,j} \frac{1}{N^2} + \mathcal{O} \Big( \frac{1}{N^3} \Big)
\end{equation}
for some constants $a_{\ell,j}$, $b_{\ell,j}$, and $c_{\ell,j}$. Since $\displaystyle \frac{1}{N_0} = \frac{n_0}{4N}$, and
\begin{equation} \label{eq:1over N}
	\frac{1}{N}
	= \frac{4}{n-n_0}
	= \frac{4}{n} \frac{1}{1-\frac{n_0}{n}}
	= \frac{4}{n} \Big( 1 + \frac{n_0}{n} + \Big( \frac{n_0}{n} \Big)^2 \Big) 
	+ \mathcal{O} \Big( \frac{1}{n^4} \Big),
\end{equation}
use of \eqref{eq:Hnjlasymp} in \eqref{eq:Hnell} implies \eqref{eq:Rn12asymp} for $\ell = 1,2$.

\subsection{Asymptotic behavior of $R_{n,3}$}

Here we study $R_{n,3}$ and prove the asymptotic relation \eqref{eq:Rn3asymp}. To this end,
we employ \eqref{eq:parfrac} for the parameters $x=k$, $a = \dfrac{\pi^2}{3n^2}$, $b = N^2- \dfrac{\pi^2}{3n^2}N^4$ and
then apply \eqref{eq:sum1} to deduce
\begin{align*}
	R_{n,3}
	& = \frac{\psi(N+1+\sqrt{B_N}) - \psi(1+\sqrt{B_N}) - \psi(N+1-\sqrt{B_N}) +  \psi(1-\sqrt{B_N})}{2A_N\sqrt{B_N}}
	\\
	& + i 
	\frac{\psi(N+1+i\sqrt{C_N}) - \psi(1+i\sqrt{C_N}) - \psi(N+1-i\sqrt{C_N}) + \psi(1-i\sqrt{C_N})}{2A_N\sqrt{C_N}}.
\end{align*}
We then use \eqref{eq:relations} and the $\pi$-periodicity of the cotangent function to obtain
\begin{align*}
	R_{n,3}
	& = \frac{1}{2A_N\sqrt{B_N}} 
	\Big(
		\psi(\sqrt{B_N}+N) - \psi(\sqrt{B_N}-N) + \frac{1}{\sqrt{B_N}+N} - \frac{1}{\sqrt{B_N}}
	\Big)
	\\
	& + \frac{i}{2A_N\sqrt{C_N}} 
	\Big(
		\psi(N+i\sqrt{C_N}) - \psi(N-i\sqrt{C_N})
	\\	
	& 	\hspace{5cm} - \frac{2i\sqrt{C_N}}{C_N+N^2} - \frac{1}{i\sqrt{C_N}} + \pi \cot \pi i \sqrt{C_N}
	\Big).
\end{align*}
Use of \eqref{eq:ABCestimate} for $k = N$ shows that the terms, aside from those involving the digamma and the cotangent functions,
are of size $\displaystyle \mathcal{O}(\frac{1}{N^2})$. For the cotangent term, employing \eqref{eq:ABCestimate} for $k = N$ we obtain
\[
	\pi \cot \pi i \sqrt{C_N}
	= -i \pi - i \pi \frac{2 e^{-2\pi\sqrt{C_N}}}{1-e^{-2\pi\sqrt{C_N}}}
	= - i \pi + \mathcal{O} (e^{-\pi N}).
\]
Accordingly,
\begin{align*}
	R_{n,3}
	& = \frac{\psi(\sqrt{B_N}+N) - \psi(\sqrt{B_N}-N)}{2A_N\sqrt{B_N}} 
	\\
	& \hspace{1cm} + i \frac{\psi(N+i\sqrt{C_N}) - \psi(N-i\sqrt{C_N}) - i\pi }{2A_N\sqrt{C_N}}
	+ \mathcal{O} \Big( \frac{1}{N^2} \Big).
\end{align*}
Then we use the asymptotic expansion \eqref{eq:asymp} (in the form $\psi(z) = \log z + \mathcal{O}(\frac{1}{z})$)
along with \eqref{eq:ABCestimate} for $k = N$ to get
\begin{align*}
	R_{n,3}
	& = \frac{1}{2A_N\sqrt{B_N}} \log(\frac{\sqrt{B_N}+N}{\sqrt{B_N}-N}) 
	+ \frac{1}{2A_N\sqrt{C_N}} \Big( \pi -2 \arctan\big(\frac{\sqrt{C_N}}{N}\big) \Big)	
	+ \mathcal{O} \Big( \frac{1}{N^2} \Big)
	\\
	& = \frac{1}{2N} \frac{1}{A_N} \frac{N}{\sqrt{B_N}} \log \Big( \frac{1+\frac{N}{\sqrt{B_N}}}{1-\frac{N}{\sqrt{B_N}}} \Big)
	- \frac{1}{N} \frac{1}{A_N} \frac{\arctan\big(\frac{\sqrt{C_N}}{N}\big)}{\frac{\sqrt{C_N}}{N}} 
	+ \dfrac{\pi}{2} \frac{1}{A_N \sqrt{C_N}} 
	+ \mathcal{O} \Big( \frac{1}{N^2} \Big)
	\\ 
	& = R_{n,3,1} - R_{n,3,2} + R_{n,3,3} 
	+ \mathcal{O}\Big( \frac{1}{N^2} \Big),
	\quad \text{say.}
	\nonumber
\end{align*}
By use of \eqref{eq:Hn1pre}, \eqref{eq:Hn2pre}, and \eqref{eq:Hndecomppart}, we have
\begin{equation*}
	R_{n,3,1} = \frac{1}{2N} \alpha_{N,8} + \mathcal{O} \Big( \frac{1}{N^2} \Big),
	\
	R_{n,3,2} = \frac{1}{N} \alpha_{N,9} + \mathcal{O} \Big( \frac{1}{N^2} \Big),
	\
	R_{n,3,3} = \frac{\pi}{2N} \alpha_{N,11} + \mathcal{O} \Big( \frac{1}{N^2} \Big)
\end{equation*}
so that
\begin{equation} \label{eq:Rn3simple}
	R_{n,3} = \frac{\alpha_{N,8} - 2\alpha_{N,9} + \pi \alpha_{N,11}}{2} \, \frac{1}{N}+ \mathcal{O} \Big( \frac{1}{N^2} \Big).
\end{equation}
Using \eqref{eq:1over N} in \eqref{eq:Rn3simple}, we conclude
\begin{align} \label{eq:Rn3asympcomp}
	R_{n,3} = 2(\alpha_{N,8} - 2\alpha_{N,9} + \pi \alpha_{N,11}) \frac{1}{n} + \mathcal{O} \Big( \frac{1}{n^2} \Big).
\end{align}
Noting that $\alpha_{N,j}$, $\beta_{N,j}$, and $\gamma_{N,j}$ are in fact constants, 
\begin{align*}
	\alpha_{N,8}
	& = \frac{\pi}{2 \sqrt{6}} \frac{1}{\mathcal{A}_N \sqrt{1 + \mathcal{A}_N}}
	\log \Big( \frac{2 \sqrt{6} \sqrt{1 + \mathcal{A}_N} + \pi}{2 \sqrt{6} \sqrt{1 + \mathcal{A}_N} - \pi } \Big)
	\\
	\alpha_{N,9}
	& = \frac{2 \sqrt{6}}{\sqrt{48-\pi^2}} \frac{\sqrt{1 + \mathcal{A}_N}}{\mathcal{A}_N}
	\arctan \Big( \frac{\sqrt{48-\pi^2}}{2 \sqrt{6} \sqrt{1 + \mathcal{A}_N}} \Big)
	\\
	\alpha_{N,11}
	& = \frac{2 \sqrt{6}}{\sqrt{48-\pi^2}} \frac{\sqrt{1 + \mathcal{A}_N}}{\mathcal{A}_N}
\end{align*}
where
\[
	\mathcal{A}_N = \sqrt{1+ \frac{\pi^2}{12} (1 - \frac{\pi^2}{48})},
\]
\eqref{eq:Rn3asympcomp} delivers \eqref{eq:Rn3asymp} with $\beta_3 = 2(\alpha_{N,8} - 2\alpha_{N,9} + \pi \alpha_{N,11})$.

\subsection{Asymptotic behavior of $R_{n,4}$}
Here we consider $R_{n,4}$ in \eqref{eq:Rn34} and prove the asymptotic relation \eqref{eq:Rn4asymp}. First we observe that
\[
	\alpha_{k,11}
	= \frac{\sqrt{1+\mathcal{A}_k}}{\sqrt{2}\sqrt{1-a_k^2}\mathcal{A}_k}
\]
is bounded for $1 \le k \le N$ so that the use of \eqref{eq:ak} in \eqref{eq:Hndecomppart} entails
\begin{align}
	R_{n,4}
	& = \sum_{k=1}^{N} \frac{1}{k} \alpha_{k,11}
	+ \frac{1}{N_0} \sum_{k=1}^{N} \frac{1}{k} \alpha_{k,11} \beta_{k,11}
	+ \frac{1}{N_0^2} \sum_{k=1}^{N} \frac{1}{k} \alpha_{k,11} \gamma_{k,11}
	+ \mathcal{O} \Big( \frac{\log N}{N_0^3} \Big)
	\nonumber
	\\
	& = \sum_{k=1}^{N} \frac{1}{k} \alpha_{k,11}
	+ \frac{\pi}{4\sqrt{3}}\frac{1}{N_0} \sum_{k=1}^{N} \frac{1}{N} \alpha_{k,11} \frac{\beta_{k,11}}{a_k}
	+ \frac{\pi}{4\sqrt{3}} \frac{1}{N_0^2} \sum_{k=1}^{N} \frac{1}{N} \alpha_{k,11} \frac{\gamma_{k,11}}{a_k}
	\nonumber
	\\
	& + \mathcal{O} \Big( \frac{\log N}{N_0^3} \Big)
	\nonumber
	\\
	& = R_{n,4,1}
	+ \frac{\pi}{4\sqrt{3}} \frac{1}{N_0} R_{n,4,2} + \frac{\pi}{4\sqrt{3}} \frac{1}{N_0^2} R_{n,4,3}
	+ \mathcal{O} \Big( \frac{\log N}{N_0^3} \Big),
	\quad
	\text{say.}
	\label{eq:Rn4}
\end{align}
For $R_{n,4,1}$, we note that
\begin{equation} \label{eq:alphak11}
	\alpha_{k,11}
	= \mu\big( \frac{k}{N} \big)
\end{equation}
where
\[
	\mu(x)
	= \frac{\sqrt{1+\sqrt{1+\frac{\pi^2}{12}x^2\Big( 1- \frac{\pi^2}{48}x^2\Big)}}}
	{\sqrt{2}\sqrt{1-\frac{\pi^2}{48}x^2}\sqrt{1+\frac{\pi^2}{12}x^2\Big( 1- \frac{\pi^2}{48}x^2\Big)}}
\]
The function $\mu$ admits a convergent Taylor series expansion valid in an interval containing $[-1,1]$  which is
\begin{equation} \label{eq:alphak11Taylor}
	\mu(x)
	= 1
	- \frac{\pi^2}{48} x^2
	+ \frac{11\pi^4}{4608} x^4
	- \frac{41\pi^6}{221184} x^6
	+ \mathcal{O}(x^8).
\end{equation}
Therefore, the function
\[
	\mu_1(x) = \frac{\mu(x)-1}{x}
\]
is smooth on the interval $[0,1]$. This motivates us to write
\begin{align*}
	R_{n,4,1}
	& = \sum_{k=1}^{N} \frac{1}{k} \alpha_{k,11}
	= \sum_{k=1}^{N} \frac{1}{k} \mu \big( \frac{k}{N}\big)
	= \sum_{k=1}^{N} \frac{1}{k} + \sum_{k=1}^{N} \frac{1}{N}\frac{\mu \big( \frac{k}{N}\big)-1}{\frac{k}{N}}
	\\
	& = \sum_{k=1}^{N} \frac{1}{k} + \sum_{k=1}^{N} \frac{1}{N} \mu_1 \big( \frac{k}{N}\big).
\end{align*}
By the well-known formula (see, for example, \cite[Formulas 6.3.2 and 6.3.18]{AS})
\[
	\sum_{k=1}^N \frac{1}{k}
	= \log N + \gamma + \frac{1}{2N} - \frac{1}{12 N^2}  + \mathcal{O} \Big( \frac{1}{N^4} \Big),
\]
and the application of  Theorem C to the second sum giving
\[
	\sum_{k=1}^{N} \frac{1}{N} \mu_1 \big( \frac{k}{N}\big)
	 = d_1 + d_2 \frac{1}{N} + d_3 \frac{1}{N^2} + \mathcal{O} \Big( \frac{1}{N^3} \Big)
\]
for some constants $d_j$, we have
\begin{equation} \label{eq:Rn41}
	R_{n,4,1}
	= \log N + a_{4,1} + b_{4,1} \frac{1}{N} + c_{4,1} \frac{1}{N^2} + \mathcal{O} \Big( \frac{1}{N^3} \Big)
\end{equation}
for some constants $a_{4,1}$, $b_{4,1}$, and $c_{4,1}$. As for $R_{n,4,2}$ and $R_{n,4,3}$,
we note that $\beta_{k,11}$ and $\gamma_{k,11}$ contain $a_k$ (in fact $a_k^2$) as a factor:
\begin{equation} \label{eq:betak11overaksq}
	\frac{\beta_{k,11}}{a_k^2}
	= \frac{2(1-2a_k^2)(2+\mathcal{A}_k)}{\mathcal{A}_k^2(1+\mathcal{A}_k)} - \frac{1}{1-a_k^2},
\end{equation}
and
\begin{align}
	\frac{\gamma_{k,11}}{a_k^2}
	& = \frac{2(8a_k^6+4a_k^4+10a_k^2-3)}{\mathcal{A}_k^4}
	- \frac{8a_k^6+4a_k^4+10a_k^2-3}{\mathcal{A}_k^3(1+\mathcal{A}_k)}
	- \frac{2a_k^2(2a_k^2-1)^2}{\mathcal{A}_k^2(1+\mathcal{A}_k)^2}
	\nonumber
	\\
	&
	- \frac{2a_k^2(2a_k^2-1)}{\mathcal{A}_k(1+\mathcal{A}_k)(1-a_k^2)}
	+ \frac{4a_k^2(2a_k^2-1)}{\mathcal{A}_k^2(1-a_k^2)}
	+ \frac{3}{2(1-a_k^2)^2}.
	\label{eq:gammak11overaksq}
\end{align}
The importance of this observation is that both $\displaystyle \frac{\beta_{k,11}}{a_k}$ and
$\displaystyle \frac{\gamma_{k,11}}{a_k}$ are of the form $\displaystyle \mu\big( \frac{k}{N} \big)$ for smooth functions $\mu$
on $[0,1]$. Since $\alpha_{k,11}$ also has the same form, we conclude for $j=2,3$ that
\[
	R_{n,4,j}
	= \sum_{k=1}^{N} \frac{1}{N} \mu_j \big( \frac{k}{N} \big)
\]
for appropriately defined smooth functions $\mu_j$ on the interval $[0,1]$. Therefore Theorem C applies
to both of these sums to yield for $j=2,3$
\begin{equation} \label{eq:Rn4j}
	R_{n,4,j}
	= a_{4,j} + b_{4,j} \frac{1}{N} + c_{4,j} \frac{1}{N^2} + \mathcal{O} \Big( \frac{1}{N^3} \Big)
\end{equation}
for some constants $a_{4,j}$, $b_{4,j}$, and $c_{4,j}$. Using \eqref{eq:Rn41} and \eqref{eq:Rn4j} in \eqref{eq:Rn4} and then making use of 
\eqref{eq:1overN0} and \eqref{eq:1over N}, we deduce \eqref{eq:Rn4asymp}.

\subsection{Asymptotic behavior of $R_{n,5}$}

Here we prove \eqref{eq:Rn5asymp} for the asymptotic behavior of $R_{n,5}$. To this end, we first show the following.
\begin{Lem} \label{lemma:Rn5simp1}
As $n \to \infty$, we have
\begin{equation} \label{eq:Rn5simple2}
	R_{n,5}
	= \sum_{1 \le k \le \log N \atop 0 \le p \le \log N} \frac{1}{k} \, \alpha_{k,11} \, e^{-2\pi \sqrt{C_k}(p+1)}
	+ \mathcal{O} \Big( \frac{1}{N^2} \Big).
\end{equation}
\end{Lem}
\begin{proof}
We use \eqref{eq:ak} in \eqref{eq:Hndecomppart} to write
\begin{multline}
	R_{n,5}
	= \sum_{k=1}^N \frac{1}{A_k\sqrt{C_k}} \, \frac{e^{-2\pi \sqrt{C_k}}}{1-e^{-2\pi \sqrt{C_k}}}
	= \sum_{k=1}^N \frac{1}{k} \, \alpha_{k,11} \, \frac{e^{-2\pi \sqrt{C_k}}}{1-e^{-2\pi \sqrt{C_k}}}
	\label{eq:Hndecomppartr}
	\\
	+ \frac{\pi}{4 \sqrt{3}} \sum_{k=1}^N \alpha_{k,11}
	\Big\{
		\frac{1}{NN_0} \frac{\beta_{k,11}}{a_k} 
		+ \frac{1}{NN_0^2} \frac{\gamma_{k,11}}{a_k}
		+ \frac{1}{k} \, \mathcal{O} \Big( \frac{1}{N_0^3} \Big)
	\Big\}	 \frac{e^{-2\pi \sqrt{C_k}}}{1-e^{-2\pi \sqrt{C_k}}}.
\end{multline}
From \eqref{eq:alphak11}, \eqref{eq:betak11overaksq}, and \eqref{eq:gammak11overaksq}, 
we see that $\alpha_{k,11}$, $\dfrac{\beta_{k,11}}{a_k}$, and $\dfrac{\gamma_{k,11}}{a_k}$ are bounded
for $1 \le k \le N$. Moreover, since $\dfrac{k}{2} < \dfrac{\sqrt{69}}{10} k < \sqrt{C_k}$,
\[
	0 < \sum_{k=1}^N \frac{1}{k} \frac{e^{-2\pi \sqrt{C_k}}}{1-e^{-2\pi \sqrt{C_k}}}
	< \sum_{k=1}^N \frac{e^{-2\pi \sqrt{C_k}}}{1-e^{-2\pi \sqrt{C_k}}}
	< \sum_{k=1}^{\infty} \frac{e^{-\pi k}}{1-e^{-\pi k}}
	< \infty.
\]
Therefore \eqref{eq:Hndecomppartr} implies
\[
	R_{n,5}
	= \sum_{k=1}^N \frac{1}{k} \, \alpha_{k,11} \, \frac{e^{-2\pi \sqrt{C_k}}}{1-e^{-2\pi \sqrt{C_k}}}
	+ \mathcal{O} \Big( \frac{1}{NN_0} \Big) + \mathcal{O} \Big( \frac{1}{NN_0^2} \Big) + \mathcal{O} \Big( \frac{1}{N_0^3} \Big)
\]
which, in turn, gives
\begin{equation} \label{eq:Rn5simple}
	R_{n,5}
	= \sum_{k=1}^N \frac{1}{k} \, \alpha_{k,11} \, \frac{e^{-2\pi \sqrt{C_k}}}{1-e^{-2\pi \sqrt{C_k}}}
	+ \mathcal{O} \Big( \frac{1}{N^2} \Big).
\end{equation}
Since $\dfrac{k}{2}< \sqrt{C_k}$, we have
\begin{align}
	0 <
	\sum_{k > \log N} \frac{1}{k} \, \frac{e^{-2\pi \sqrt{C_k}}}{1-e^{-2\pi \sqrt{C_k}}}
	& < \sum_{k > \log N} \frac{1}{k} \frac{e^{-\pi k}}{1-e^{-\pi k}} 
	< \sum_{k > \log N} \frac{e^{-\pi k}}{1-e^{-\pi}}
	\label{eq:simpleest1}
	\\
	& \le \frac{e^{-\pi \log N}}{(1-e^{-\pi})^2}  
	= \frac{N^{-\pi}}{(1-e^{-\pi})^2}.
	\nonumber
\end{align}
Accordingly, since $\alpha_{k,11}$ is bounded, using \eqref{eq:simpleest1} in \eqref{eq:Rn5simple}, we obtain
\begin{equation} \label{eq:Rn5simple1}
	R_{n,5}
	= \sum_{1 \le k \le \log N} \frac{1}{k} \, \alpha_{k,11} \, \frac{e^{-2\pi \sqrt{C_k}}}{1-e^{-2\pi \sqrt{C_k}}}
	+ \mathcal{O} \Big( \frac{1}{N^2} \Big).
\end{equation}
Note further that
\begin{align*}
	\sum_{1 \le k \le \log N} \frac{1}{k} \, \alpha_{k,11} \, \frac{e^{-2\pi \sqrt{C_k}}}{1-e^{-2\pi \sqrt{C_k}}}
	& - \sum_{1 \le k \le \log N \atop 0 \le p \le \log N} \frac{1}{k} \, \alpha_{k,11} \, e^{-2\pi \sqrt{C_k}(p+1)}
	\\
	&= \sum_{1 \le k \le \log N \atop p > \log N} \frac{1}{k} \, \alpha_{k,11} \, e^{-2\pi \sqrt{C_k}(p+1)} 
\end{align*}
and, using $\dfrac{k}{2}< \sqrt{C_k}$,
\begin{align*}
	0 & < \sum_{1 \le k \le \log N \atop p > \log N} \frac{1}{k} \, e^{-2\pi \sqrt{C_k}(p+1)} 
	< \sum_{1 \le k \le \log N \atop p > \log N} e^{-\pi k(p+1)} 
	\le \sum_{1 \le k \le \log N} \frac{e^{-\pi k (1+ \log N)}}{1-e^{-\pi k}}
	\\
	& 
	< \sum_{k=1}^{\infty} \frac{e^{-\pi k (1+ \log N)}}{1-e^{-\pi}}
	= \frac{1}{1-e^{-\pi}} \frac{(eN)^{-\pi}}{1-(eN)^{-\pi}}
	\le \frac{e^{-\pi}}{(1-e^{-\pi})^2} \, N^{-\pi}.
\end{align*}
Therefore, since $\alpha_{k,11}$ is bounded, \eqref{eq:Rn5simple1} gives \eqref{eq:Rn5simple2}. This completes the proof.
\end{proof}

To further simply the relation given in Lemma \ref{lemma:Rn5simp1}, we estimate $\alpha_{k,11}$ and the term in the exponent in \eqref{eq:Rn5simple2}, 
and prove the following.

\begin{Lem}
As $n \to \infty$, we have
\begin{equation} \label{eq:Rnsimplelog}
	R_{n,5}
	= \sum_{1 \le k \le \log N \atop  0 \le p \le \log N} \frac{1}{k} \, e^{-2\pi k(p+1)}
	+ \mathcal{O} \Big( \frac{1}{N^2} \Big).
\end{equation}
\end{Lem}
\begin{proof}
From \eqref{eq:alphak11} and \eqref{eq:alphak11Taylor}, we have for $1 \le k \le \log N$
\begin{equation} \label{eq:alphak11TaylorkN}
	\alpha_{k,11} = 1 - \frac{\pi^2}{48} \frac{k^2}{N^2} + \mathcal{O} \Big( \frac{\log^4 N}{N^4} \Big).
\end{equation}
On the other hand, from \eqref{eq:SqrtCoN}, we have
\begin{equation} \label{eq:SqrtCoNr}
	-2\pi \sqrt{C_k}
	= -2\pi \frac{\sqrt{2}\sqrt{1-a_k^2}}{\sqrt{1+\mathcal{A}_k}} \, k
	\Big\{
		1
		+ \beta_{k,6} \frac{1}{N_0}
		+ \gamma_{k,6} \frac{1}{N_0^2}
		+ \mathcal{O}\Big( \frac{1}{N_0^3} \Big)
	\Big\}.
\end{equation}
Replacing $\dfrac{k}{N}$ by $x$ in $\dfrac{\sqrt{2}\sqrt{1-a_k^2}}{\sqrt{1+\mathcal{A}_k}}$, we obtain the function
\[
	\mu(x)
	= \frac{\sqrt{2}\sqrt{1-\frac{\pi^2}{48}x^2}}
	{\sqrt{1+\sqrt{1+\frac{\pi^2}{12}x^2\Big( 1- \frac{\pi^2}{48}x^2\Big)}}}
\]
for which Taylor's theorem with remainder delivers
\[
	\mu(x) = 1 - \frac{\pi^2}{48} x^2 + \mathcal{O}(x^4),
	\qquad
	\text{for}
	\quad
	x \in [-\frac{4 \sqrt{3}}{\pi}, \frac{4 \sqrt{3}}{\pi}]. 
\]
Therefore, for $1 \le k \le \log N$, we have
\[
	-2 \pi \frac{\sqrt{2}\sqrt{1-a_k^2}}{\sqrt{1+\mathcal{A}_k}}
	= \Big\{ -2\pi + \frac{\pi^3}{24} \frac{k^2}{N^2} + \mathcal{O} \Big( \frac{\log^4N}{N^4} \Big)\Big\}.
\]
Using this in \eqref{eq:SqrtCoNr} and recalling \eqref{eq:ak}, we obtain for $1 \le k \le \log N$
\begin{multline} \label{eq:Ckproduct}
	-2\pi \sqrt{C_k} \, (p+1)
	= \Big\{ -2\pi + \frac{\pi^3}{24} \frac{k^2}{N^2} + \mathcal{O} \Big( \frac{\log^4N}{N^4} \Big)\Big\}
	\\
	\times
	\Big\{
		1
		+ \frac{\pi^2}{48} \frac{\beta_{k,6}}{a_k^2} \frac{k^2}{N^2N_0}
		+ \frac{\pi^2}{48} \frac{\gamma_{k,6}}{a_k^2} \frac{k^2}{N^2N_0^2}
		+ \mathcal{O}\Big( \frac{1}{N_0^3} \Big)
	\Big\}
	k(p+1).
\end{multline}
Note that
\[
	\frac{\beta_{k,6}}{a_k^2}
	= \frac{2 (1-2a_k^2)}{\mathcal{A}_k (1+\mathcal{A}_k)} + \frac{1}{1-a_k^2}
\]
\begin{align*}
	\frac{\gamma_{k,6}}{a_k^2}
	& = \frac{1}{\mathcal{A}_k^2(1+\mathcal{A}_k)}
	\Big( \frac{6a_k^2(2a_k^2-1)^2}{1+\mathcal{A}_k} - \frac{(2a_k^2-3)(12a_k^4-1)}{\mathcal{A}_k} \Big)
	\\
	& \hspace{3cm}+ \frac{2a_k^2}{\mathcal{A}_k (1+\mathcal{A}_k)} \frac{1-2a_k^2}{1-a_k^2}
	+ \frac{2a_k^2-3}{2 (1-a_k^2)^2}
\end{align*}
are bounded for $1 \le k \le N$ and therefore,
for $1 \le k \le \log N$ and $0 \le p \le \log N$, we have 
\begin{equation} \label{eq:betak6powers}
	\beta_{k,6} \, k^j \, (p+1)
	= \frac{\pi^2}{48} \, \frac{\beta_{k,6}}{a_k^2} \, \frac{k^{j+2} \, (p+1)}{N^2}
	= \mathcal{O} \Big( \frac{\log^{j+3} N}{N^2} \Big),
	\quad
	j \in \mathbb{N},
\end{equation}
and this also holds when $\beta_{k,6}$ is replaced with $\gamma_{k,6}$. Note also that
\begin{equation} \label{eq:logj}
	\exp \Big( \frac{\log^{j} N}{N^{\alpha} N_0^{\beta}} \Big)
	= 1 + \mathcal{O} \Big( \frac{\log^{j} N}{N^{\alpha+\beta}} \Big), 
	\qquad
	\text{for }
	j,\alpha,\beta \ge 0, \ 
	\alpha + \beta > 0,
\end{equation}
where we have used $e^x = 1+x + \mathcal{O}(x^2) = 1 + \mathcal{O}(x)$ for $x=o(1)$.
Therefore, in light of \eqref{eq:Ckproduct}, use of \eqref{eq:betak6powers} and \eqref{eq:logj}
implies for $1 \le k \le \log N$ and $0 \le p \le \log N$
\begin{align} 
	e^{-2\pi \sqrt{C_k}(p+1)}
	&= e^{-2\pi k(p+1)} e^{\frac{\pi^3}{24} \frac{k^3}{N^2}(p+1)} \Big( 1 + \mathcal{O} \Big( \frac{\log^4 N}{N^3} \Big) \Big)
	\nonumber
	\\
	& = e^{-2\pi k(p+1)} \Big(1 + \frac{\pi^3}{24} \frac{k^3}{N^2}(p+1) + \mathcal{O} \big( \frac{\log^8 N}{N^4} \big) \Big)
	\Big( 1 + \mathcal{O} \Big( \frac{\log^4 N}{N^3} \Big) \Big)
	\nonumber
	\\
	& = e^{-2\pi k(p+1)} \Big\{ 1+ \frac{\pi^3}{24} \frac{k^3}{N^2}(p+1) + \mathcal{O} \Big( \frac{\log^4 N}{N^3} \Big) \Big\}.
	\label{eq:Cksimplest}
\end{align}
Finally, using \eqref{eq:alphak11TaylorkN} and \eqref{eq:Cksimplest}, we obtain for $1 \le k \le \log N$ and $0 \le p \le \log N$
\begin{align} 
	\frac{1}{k} \, \alpha_{k,11} & \, e^{-2\pi \sqrt{C_k}(p+1)}
	\label{eq:Rn5coef}
	\\
	& = \frac{1}{k} \, e^{-2\pi k(p+1)}
	\Big\{ 1 - \frac{\pi^2}{48} \frac{k^2}{N^2} + \frac{\pi^3}{24} \frac{k^3}{N^2}(p+1) + \mathcal{O} \Big( \frac{\log^4 N}{N^3} \Big) \Big\}.
	\nonumber
\end{align}
Moreover, since the double series
\[
	\sum_{k \ge 1 \atop p \ge 0} e^{-2\pi k(p+1)} \, k^{\alpha} \, (1+p)^{\beta}
\]
converges for any $\alpha,\beta \in \mathbb{R}$, use of \eqref{eq:Rn5coef} entails
\begin{align*}
	\sum_{1 \le k \le \log N \atop  0 \le p \le \log N} \frac{1}{k} \, \alpha_{k,11} \, e^{-2\pi \sqrt{C_k}(p+1)}
	& = \sum_{1 \le k \le \log N \atop  0 \le p \le \log N} \frac{1}{k} \, e^{-2\pi k(p+1)}
	+ \mathcal{O} \Big( \frac{1}{N^2} \Big),
\end{align*}
and therefore \eqref{eq:Rn5simple2} yields \eqref{eq:Rnsimplelog} and this completes the proof.
\end{proof}

Finally, we establish the following relation.

\begin{Lem}
As $n \to \infty$, there holds
\begin{equation} \label{eq:Rnsimplest1}
	R_{n,5}
	= \sum_{k=1}^{\infty} \frac{1}{k} \, \frac{e^{-2\pi k}}{1-e^{-2\pi k}}
	+ \mathcal{O} \Big( \frac{1}{n^2} \Big).
\end{equation}
\end{Lem}
\begin{proof}
Since
\[
	\sum_{1 \le k \le \log N} \frac{1}{k} \, \frac{e^{-2\pi k}}{1-e^{-2\pi k}}
	- \sum_{1 \le k \le \log N \atop 0 \le p \le \log N} \frac{1}{k} \, e^{-2\pi k(p+1)}
	= \sum_{1 \le k \le \log N \atop p > \log N} \frac{1}{k} \, e^{-2\pi k(p+1)} 
\]
and
\begin{align*}
	0 < \sum_{1 \le k \le \log N \atop p > \log N} \frac{1}{k} \, e^{-2\pi k(p+1)} 
	& < \sum_{1 \le k \le \log N \atop p > \log N} e^{-2\pi k(p+1)} 
	\le \sum_{1 \le k \le \log N} \frac{e^{-2\pi k(1+\log N)}}{1-e^{-2\pi k}}
	\\
	& < \sum_{k=1}^{\infty} \frac{e^{-2\pi k(1+\log N)}}{1-e^{-2\pi}}
	\le \frac{e^{-2\pi}}{(1-e^{-2\pi})^2} \,N^{-2\pi},
\end{align*}
\eqref{eq:Rnsimplelog} implies
\begin{equation} \label{eq:Rnsimplelog1}
	R_{n,5}
	= \sum_{1 \le k \le \log N} \frac{1}{k} \, \frac{e^{-2\pi k}}{1-e^{-2\pi k}}
	+ \mathcal{O} \Big( \frac{1}{N^2} \Big).
\end{equation}
Further, since
\[
	0 < \sum_{ k > \log N} \frac{1}{k} \, \frac{e^{-2\pi k}}{1-e^{-2\pi k}}
	< \sum_{ k > \log N} \frac{e^{-2\pi k}}{1-e^{-2\pi}}
	\le \frac{N^{-2\pi}}{(1-e^{-2\pi})^2}
\]
\eqref{eq:Rnsimplelog1} entails
\begin{equation} \label{eq:Rnsimplest}
	R_{n,5}
	= \sum_{k=1}^{\infty} \frac{1}{k} \, \frac{e^{-2\pi k}}{1-e^{-2\pi k}}
	+ \mathcal{O} \Big( \frac{1}{N^2} \Big).
\end{equation}
Therefore \eqref{eq:Rnsimplest1} follows from \eqref{eq:Rnsimplest} with the aid of \eqref{eq:1over N}. 
\end{proof}

As for the proof of \eqref{eq:Rn5asymp}, we note for $q = e^{-2\pi}$
\begin{align}
	\sum_{k=1}^{\infty} \frac{1}{k} \, \frac{e^{-2\pi k}}{1-e^{-2\pi k}}
	& = \sum_{k=1}^{\infty} \sum_{m=0}^{\infty} \frac{1}{k} \, (q^{m+1})^k
	= - \sum_{m=0}^{\infty} \sum_{k=1}^{\infty} \frac{-1}{k} \, (q^{m+1})^k
	\label{eq:last}
	\\
	& = - \sum_{m=0}^{\infty} \log(1-q^{m+1})
	= - \log( \prod_{m=0}^{\infty} (1-q^{m+1}))
	\nonumber
	\\
	& = - \log( (q;q)_{\infty})
	= - \log( q^{-\frac{1}{24}} \eta(i))
	\nonumber
\end{align}
where $(q;q)_{\infty}$ is the $q$-Pochhammer symbol and $\eta$ is the Dedekind eta-function \cite{Z}. Since 
$\displaystyle \eta(i) = \frac{\Gamma(\frac{1}{4})}{2\pi^\frac{3}{4}}$, use of \eqref{eq:last} in \eqref{eq:Rnsimplest1} proves \eqref{eq:Rn5asymp}.

\subsection{Asymptotic behavior of $Q_n$}
Here we study $Q_n$ in \eqref{eq:Rn5Qn} and establish \eqref{eq:Qnasymp}. To this end, we use partial fractions to write
\begin{equation} \label{eq:Qnparfrac}
	Q_{n}
	= \dfrac{\pi}{2\sqrt{3}n} \sum_{k=1}^{N} \Big( \frac{1}{k + \frac{\sqrt{3}n}{\pi}} - \frac{1}{k - \frac{\sqrt{3}n}{\pi}} \Big)
	+ \sum_{k=1}^{N} \frac{1}{k^2}
	= \dfrac{\pi}{2\sqrt{3}n} \, Q_{n,1} + \sum_{k=1}^{N} \frac{1}{k^2},
\end{equation}
say. For $Q_{n,1}$, we utilize \eqref{eq:sum1} to obtain
\begin{align*}
	Q_{n,1}
	= \psi(N+1+\frac{\sqrt{3}n}{\pi}) - \psi(1+\frac{\sqrt{3}n}{\pi}) - \psi(N+1-\frac{\sqrt{3}n}{\pi}) + \psi(1-\frac{\sqrt{3}n}{\pi}).
\end{align*}
Next we employ \eqref{eq:formula} to write
\[
	\psi(N+1+\frac{\sqrt{3}n}{\pi}) = \psi(\frac{\sqrt{3}n}{\pi}+N) + \frac{1}{\frac{\sqrt{3}n}{\pi}+N},
\]
\[	
	\psi(N+1-\frac{\sqrt{3}n}{\pi}) = \psi(\frac{\sqrt{3}n}{\pi}-N) + \pi \cot \pi (\frac{\sqrt{3}n}{\pi}-N),
\]
\[
	\psi(1+\frac{\sqrt{3}n}{\pi}) - \psi(1-\frac{\sqrt{3}n}{\pi}) = \frac{1}{\frac{\sqrt{3}n}{\pi}} - \pi \cot \pi \frac{\sqrt{3}n}{\pi},
\]
and use the $\pi$-periodicity of the cotangent function to have
\[
	Q_{n,1}
	= \psi(\frac{\sqrt{3}n}{\pi}+N) - \psi(\frac{\sqrt{3}n}{\pi}-N) + \mathcal{O} \Big( \frac{1}{n} \Big).
\]
Then we apply \eqref{eq:asymp} (in the form $\psi(z) = \log z + \mathcal{O}(\frac{1}{z})$) to deduce
\begin{equation} \label{eq:Qnlog}
	Q_{n,1}
	= \log\Big(\frac{\frac{\sqrt{3}n}{\pi}+N}{\frac{\sqrt{3}n}{\pi}-N}\Big)
	+ \mathcal{O} \Big( \frac{1}{n} \Big).
\end{equation}
Since
\[
	\frac{\frac{\sqrt{3}n}{\pi}+N}{\frac{\sqrt{3}n}{\pi}-N} 
	= \frac{\frac{\sqrt{3}n}{\pi}+\frac{n-n_0}{4}}{\frac{\sqrt{3}n}{\pi}-\frac{n-n_0}{4}}
	= \frac{4\sqrt{3}+\pi}{4\sqrt{3}-\pi} \frac{1- \frac{\pi n_0}{4\sqrt{3}+\pi}\frac{1}{n}}{1+ \frac{\pi n_0}{4\sqrt{3}-\pi}\frac{1}{n}},
\]
\eqref{eq:Qnlog} entails
\begin{equation} \label{eq:Qnlogsimple}
	Q_{n,1}
	= \log\Big(\frac{4\sqrt{3}+\pi}{4\sqrt{3}-\pi}\Big)
	+ \mathcal{O} \Big( \frac{1}{n} \Big).
\end{equation}
Finally, from \cite[Formulas 6.4.3 \& 6.4.12]{AS}, we have
\begin{equation} \label{eq:ksq}
	\sum_{k=1}^{N} \frac{1}{k^2}
	= \frac{\pi^2}{6} - \frac{1}{N} + \mathcal{O} \Big( \frac{1}{N^2}\Big)
	= \frac{\pi^2}{6} - \frac{4}{n} + \mathcal{O} \Big( \frac{1}{n^2}\Big).
\end{equation}
Accordingly, use of \eqref{eq:Qnlogsimple} and \eqref{eq:ksq} in \eqref{eq:Qnparfrac} proves \eqref{eq:Qnasymp}.

%

\section*{Acknowledgment} This research is supported by Bo\u{g}azi\c{c}i University Research Fund Grant Number 13387.

\end{document}